\documentclass[11pt,reqno]{amsart}
\usepackage{amsmath,amsfonts,amssymb}
\usepackage{amsmath,amsfonts,amssymb}
\usepackage{mathrsfs}

\def\wh{\widehat}
\def\wt{\widetilde}
\def\R{\mathbb R}
\def\C{\mathbb C}
\def\Z{\mathbb Z}

\def\A{\mathcal A}

\def\O{\mathcal O}
\def\L{\mathcal L}

\def\u{\mathbf u}
\def\e{\mathbf e}
\def\f{\mathbf f}
\def\x{\mathbf x}
\def\v{\mathbf v}
\def\a{\mathbf a}
\def\s{\mathbf s}
\def\p{\mathbf p}
\def\r{\mathbf r}

\def\c{\mathbf c}

\def\q{\mathbf q}

\def\w{\mathbf w}
\def\z{\mathbf z}
\def\n{\mathbf n}
\def\l{\mathfrak l}
\def\VV{\mathbf V}

\def\D{\mathbf D}
\def\FF{\mathbf F}
\def\G{\mathbf G}

\def\1{\mathbf 1}
\def\K{\mathcal K}
\def\M{\mathcal M}

\def\CC{\mathbf C}

\def\O{\mathcal O}

\def\bxi{\boldsymbol \xi}

\def\bphi{\boldsymbol \phi}
\def\bpsi{\boldsymbol \psi}

\def\f{\boldsymbol \varphi}
\def\bzeta{\boldsymbol \zeta}

\def\1{\bold 1}
\def\rot{\mathrm{curl}\,}
\def\div{\mathrm{div}\,}

\def\rank{\mathrm{rank}\,}

\def\eps{\varepsilon}
\def\Dom{\mathrm{Dom}\,}

\def\le{\leqslant}

\theoremstyle{plain}
\newtheorem{theorem}{Theorem}[section]
\newtheorem{proposition}[theorem]{Proposition}
\newtheorem{lemma}[theorem]{Lemma}
\newtheorem{condition}[theorem]{Condition}
\newtheorem{corollary}[theorem]{Corollary}

\theoremstyle{definition}
\newtheorem{definition}[theorem]{Definition}
\newtheorem{remark}[theorem]{Remark}
\numberwithin{equation}{section}

\usepackage[paperwidth=20.5cm,paperheight=29.7cm,width=16cm,height=24cm,top=24mm,left=22mm,headsep=11mm]{geometry}

\begin{document}

\dedicatory{To the memory of Mikhail Zaharovich Solomyak}

\title[Homogenization of a stationary Maxwell system]
{Homogenization of a stationary periodic \\ Maxwell system in a bounded domain 
\\ in the case of constant magnetic permeability}

\author{T.~A.~Suslina}

\address{St.~Petersburg State University
\\
Universitetskaya nab. 7/9
\\
199034, St. Petersburg, Russia}

\email{t.suslina@spbu.ru}



\thanks{Supported by Russian Science Foundation (project 17-11-01069)}

\begin{abstract}
In a bounded domain $\mathcal{O}\subset\mathbb{R}^3$ of class $C^{1,1}$, we consider a stationary Maxwell system with the boundary conditions of perfect conductivity. It is assumed that the magnetic permeability is given by a constant positive $(3\times 3)$-matrix $\mu_0$ and the dielectric permittivity is of the form  
$\eta(\x/\eps)$, where $\eta(\x)$ is a $(3 \times 3)$-matrix-valued function with real entries, periodic with respect to some lattice, bounded and positive definite. Here $\eps >0$ is the small parameter. Suppose that the equation involving the curl of the magnetic field intensity is homogeneous, and the right-hand side $\r$ of the second equation is a divergence-free vector-valued function of class $L_2$. It is known that, as $\eps \to 0$, the solutions of the Maxwell system, namely,  the electric field intensity $\u_\eps$, the electric displacement vector $\w_\eps$, the magnetic field intensity $\v_\eps$, and the magnetic displacement vector $\z_\eps$  weakly converge in  $L_2$ to the corresponding homogenized fields $\u_0$, $\w_0$, $\v_0$,~$\z_0$ (the solutions of the homogenized Maxwell system with effective coefficients). We improve the classical results. It is shown that  $\v_\eps$ and $\z_\eps$ converge to $\v_0$ and $\z_0$, respectively, in the $L_2$-norm, the error terms do not exceed $C \eps \|\r\|_{L_2}$. 
We also find approximations for $\v_\eps$ and $\z_\eps$ in the energy norm with error  
$C\sqrt{\eps} \|\r\|_{L_2}$. For $\u_\eps$ and $\w_\eps$ we obtain approximations in the  $L_2$-norm with error
$C\sqrt{\eps} \|\r\|_{L_2}$.
\end{abstract}

\maketitle

\section*{Introduction}

The paper concerns homogenization of periodic differential operators.
The literature on homogenization is very extensive; we  mention the books \cite{BeLPap, BaPa,Sa, ZhKO}. 

\subsection{Operator error estimates}

Let $\Gamma \subset \R^d$ be a lattice. For $\Gamma$-periodic functions in $\R^d$ we denote  
$$
f^\eps(\x):= f(\x/\eps),\quad \eps >0.
$$ 

In a series of papers \cite{BSu1,BSu2,BSu3} by Birman and Suslina, 
an operator-theoretic approach to homogenization theory was suggested and developed. 
This approach was applied to a wide class of matrix strongly elliptic second order operators $\A_\eps$ acting in $L_2(\R^d;\C^n)$ and admitting a factorization of the form 
\begin{equation}
\label{0.1}
\A_\varepsilon = b(\D)^* g^\eps(\x) b(\D).
\end{equation}
Here the matrix-valued function $g(\x)$ is bounded, positive definite, and periodic with respect to the lattice 
$\Gamma$. Next, $b(\D)$ is the matrix first order operator of the form 
$b(\D)= \sum_{j=1}^d b_j D_j$ such that its symbol has maximal rank. 
The simplest example of the operator \eqref{0.1}  is the scalar elliptic operator  
$$
\A_\eps = - \div g^\eps(\x) \nabla = \D^* g^\eps(\x)\D \text{ (the acoustics operator).}
$$
The elasticity operator also can be written in the form  \eqref{0.1}. 
In electrodynamics, the auxiliary operator 
$\A_\eps = \rot a^\eps(\x) \rot - \nabla \nu^\eps(\x) \div$ arises, 
it can be represented in the form \eqref{0.1}.

In \cite{BSu1}, it was shown that the resolvent $(\A_\eps +I)^{-1}$ 
converges in the operator norm in $L_2(\R^d;\C^n)$ to the resolvent of the \textit{effective operator} 
$\A^0= b(\D)^* g^0 b(\D)$, as  $\eps \to 0$. Here $g^0$ is a constant positive matrix called the
 \textit{effective matrix}. We have
 \begin{equation}
\label{0.2}
\| (\A_\varepsilon +I)^{-1} - (\A^0 +I)^{-1} \|_{L_2(\R^d) \to L_2(\R^d)} \le C \eps. 
\end{equation}
In \cite{BSu3}, approximation of the resolvent $(\A_\eps +I)^{-1}$ in the norm of operators acting from  $L_2(\R^d;\C^n)$ to the Sobolev space $H^1(\R^d;\C^n)$ was found:
\begin{equation}
\label{0.3}
\| (\A_\varepsilon +I)^{-1} - (\A^0 +I)^{-1} - \eps \K(\eps) \|_{L_2(\R^d) \to H^1(\R^d)} \le C \eps. 
\end{equation}
Here $\K(\eps)$ is the so called \textit{corrector}. It involves a rapidly oscillating factor, and so depends on $\eps$; herewith,
$$
\|\K(\eps)\|_{L_2 \to H^1}= O(\eps^{-1}).
$$

Estimates \eqref{0.2} and \eqref{0.3} are order-sharp. The results of such type are called  \textit{operator error estimates} in homogenization theory. A different approach to operator error estimates 
(the modified method of the first order approximation or the shift method) was suggested by Zhikov. 
In \cite{Zh1, ZhPas1}, estimates \eqref{0.2} and \eqref{0.3} for the acoustics and elasticity operators were obtained by this method. Further results are discussed in the survey \cite{ZhPas2}.

Operator error estimates were also studied for boundary value problems in a bounded domain $\O \subset \R^d$ with sufficiently smooth boundary;  
see \cite{ZhPas1, ZhPas2, Gr1, Gr2, KeLiS, PSu, Su13, Su_SIAM, Su15}.
Let $A_{D,\eps}$ and $A_{N,\eps}$ be the operators in $L_2(\O;\C^n)$ given by the expression 
$b(\D)^* g^\eps(\x) b(\D)$ with the Dirichlet or Neumann conditions on the boundary. Let $A_D^0$ and  $A_N^0$ be the corresponding effective operators. We have 
\begin{align}
\label{0.4}
& \| (A_{\flat,\varepsilon} +I)^{-1} - (A_\flat^0 +I)^{-1}  \|_{L_2(\O) \to L_2(\O)} \le C \eps, 
\\
\label{0.5}
& \| (A_{\flat,\varepsilon} +I)^{-1} - (A_\flat^0 +I)^{-1} - \eps K_\flat(\eps) \|_{L_2(\O) \to H^1(\O)} \le C \eps^{1/2}. 
\end{align}
Here $\flat = D,N$, and $K_\flat(\eps)$ is the corresponding corrector. Esrimate \eqref{0.4} is of sharp order $O(\eps)$ (as for the problem in $\R^d$).
The order of estimate \eqref{0.5} is worse than the order of  \eqref{0.3}; this is caused by the boundary influence. 

In \cite{ZhPas1}, by the shift method, estimate \eqref{0.5} and the analog of estimate \eqref{0.4} with error $O(\sqrt{\eps})$ were obtained for the acoustics and elasticity operators.   
Independently, similar results for the acoustics operator were obtained by Griso \cite{Gr1, Gr2} with the help of the unfolding method.  For the first time, sharp-order estimate \eqref{0.4} was proved in  \cite{Gr2}. 
The case of matrix elliptic operators was studied in \cite{KeLiS} (where uniformly elliptic operators  were considered under some regularity assumptions on the coefficients) and in  \cite{PSu, Su13, Su_SIAM, Su15} (where estimates \eqref{0.4} and \eqref{0.5} were obtained for the class of strongly elliptic operators described above). 

\subsection{Homogenization of the Maxwell system in $\R^3$\label{sec0.2}}

Now, we discuss the  homogenization problem for the stationary Maxwell system in $\R^3$.

Suppose that the dielectric permittivity and the magnetic permeability are given by the matrix-valued functions
$\eta^\eps(\x)$ and $\mu^\eps(\x)$, where $\eta(\x)$ and $\mu(\x)$ are bounded, positive definite, and periodic with respect to some lattice $\Gamma$. By~$J(\R^3)$ we denote the subspace  of vector-valued functions ${\mathbf f} \in L_2(\R^3; \C^3)$ such that $\div {\mathbf f} =0$ (in the sense of distributions).
Let $\u_\eps$ and $\v_\eps$ be the intensities of the electric and magnetic fields; $\w_\eps = \eta^\eps \u_\eps$
and $\z_\eps = \mu^\eps \v_\eps$ are the electric and magnetic displacement vectors.
We write the Maxwell operator~$\M_\eps$ in terms of the displacement vectors, assuming that $\w_\eps$ and $\z_\eps$ are divergence-free.
Then the operator  $\M_\eps$ acts in the space  $J(\R^3)\oplus J(\R^3)$ and is given by the expression
$$
\M_\eps = \begin{pmatrix}
0 & i \rot (\mu^\eps)^{-1} \cr - i \rot (\eta^\eps)^{-1} & 0
\end{pmatrix}
$$
on the natural domain. The operator $\M_\eps$ is selfadjoint, if $J(\R^3)\oplus J(\R^3)$ is considered as a subspace of the weighted space 
$L_2(\R^3;\C^3; (\eta^\eps)^{-1}) \oplus L_2(\R^3;\C^3; (\mu^\eps)^{-1})$. 
The point $\lambda =i$ is a regular point for the operator $\M_\eps$.

Let us discuss the question about the behavior of the resolvent $(\M_\eps - iI)^{-1}$ for small $\eps$. In other words, we are interested in the behavior of the solutions $(\w_\eps, \z_\eps)$ of the Maxwell system 
\begin{equation}
\label{0.6}
(\M_\varepsilon - iI) \begin{pmatrix} \w_\eps \cr \z_\eps \end{pmatrix} = \begin{pmatrix} \q \cr \r \end{pmatrix}, 
\quad  \q, \r \in J(\R^3;\C^3),
\end{equation}
and also in the behavior of the fields $\u_\eps = (\eta^\eps)^{-1}\w_\eps$ and $\v_\eps = (\mu^\eps)^{-1}\z_\eps$.

The homogenized Maxwell operator  $\M^0$ has the coefficients $\eta^0$ and $\mu^0$; 
it is well known that the effective matrices $\eta^0$ and $\mu^0$ are the same as for the  
scalar elliptic operators $-\div \eta^\eps \nabla$ and $-\div \mu^\eps \nabla$. Let
$(\w_0, \z_0)$ be the solution of the homogenized Maxwell system
\begin{equation*}
(\M^0 - iI) \begin{pmatrix} \w_0 \cr \z_0 \end{pmatrix} = \begin{pmatrix} \q \cr \r \end{pmatrix}. 
\end{equation*}
Let $\u_0 = (\eta^0)^{-1}\w_0$ and $\v_0 = (\mu^0)^{-1}\z_0$.
From the classical results (see, e.~g.,  \cite{BeLPap,  Sa, ZhKO}) it is known that, as $\eps \to 0$, the vector-valued functions $\u_\eps, \w_\eps, \v_\eps, \z_\eps$ 
\textit{weakly converge} in $L_2(\R^3;\C^3)$ to the corresponding homogenized fields
 $\u_0, \w_0, \v_0, \z_0$.

Operator error estimates for the Maxwell system \eqref{0.6} were studied in \cite[Chapter~7]{BSu1}, \cite[\S 14]{BSu2}, \cite[\S 22]{BSu3}, \cite{Su1,BSu4}, and \cite{Su2}.
In \cite{BSu1, BSu2, BSu3}, the case where $\mu = \1$ was considered and approximations were found 
not for all physical fields; in~\cite{Su1}, the general case was considered, but approximations were found not for all fields; in \cite{BSu4}, the problem was solved completely in the case of constant magnetic permeability; 
finally, in \cite{Su2}, a complete solution was achieved in the general case. The method was to reduce the problem to homogenization of some auxiliary second order equation. The solution of system \eqref{0.6} can be written as  $\w_\eps = \w_\eps^{(\q)}+\w_\eps^{(\r)}$, $\z_\eps = \z_\eps^{(\q)}+\z_\eps^{(\r)}$,
where $(\w_\eps^{(\q)},\z_\eps^{(\q)})$ is the solution of the system with $\r=0$, and 
$(\w_\eps^{(\r)},\z_\eps^{(\r)})$ is the solution of the system with $\q=0$. 
For instance, let us consider $(\w_\eps^{(\r)},\z_\eps^{(\r)})$. We substitute the first equation  
$\w_\eps^{(\r)} = \rot (\mu^\eps)^{-1} \z_\eps^{(\r)}$ in the second one and arrive at the following problem for 
$\z_\eps^{(\r)}$:
$$
\rot (\eta^\eps)^{-1}\rot (\mu^\eps)^{-1} \z_\eps^{(\r)} + \z_\eps^{(\r)} = i \r, \quad \div \z_\eps^{(\r)} =0. 
$$
Substituting $\f_\eps^{(\r)} = (\mu^\eps)^{-1/2}\z_\eps^{(\r)}$ and lifting the divergence-free condition, we see that $\f_\eps^{(\r)}$ is the solution of the second order elliptic equation 
\begin{equation}
\label{0.8}
\L_\eps \f_\eps^{(\r)}+ \f_\eps^{(\r)}= i (\mu^\eps)^{-1/2}\r,
\end{equation}
where
\begin{equation}
\label{0.9}
\L_\eps = (\mu^\eps)^{-1/2}\rot (\eta^\eps)^{-1}\rot (\mu^\eps)^{-1/2} - (\mu^\eps)^{1/2} \nabla \div (\mu^\eps)^{1/2}. 
\end{equation}
The field $\w_\eps^{(\r)}$ is expressed in terms of the derivatives of the solution: 
$$
\w_\eps^{(\r)} = \rot (\mu^\eps)^{-1/2} \f_\eps^{(\r)}.
$$

In the case of constant $\mu$, the operator \eqref{0.9} belongs to the class of operators~\eqref{0.1}, which allows one to apply general results of the papers \cite{BSu1, BSu2, BSu3} to equation \eqref{0.8}.
If $\mu$ is variable, this is not the case, but  it is possible to use the abstract scheme from \cite{BSu1, BSu2, BSu3} to study the operator  \eqref{0.9}; this was done in  \cite{Su1, Su2}. The result of these considerations was  approximation of the resolvent  ${(\M_\eps - iI)^{-1}}$. In contrast to the resolvent of the operator \eqref{0.1}, this resolvent has no limit in the operator norm, but it can be approximated by the sum of the resolvent $(\M^0 - iI)^{-1}$ and  some corrector of zero order (which weakly tends to zero); the error estimate is of sharp order $O(\eps)$. 
In terms of the solutions,  this implies approximations for all physical fields in the $L_2(\R^3;\C^3)$-norm with error estimates of order $O(\eps)$.
For instance, we write down the result for $\u_\eps$:
$$
\| \u_\eps - \u_0 - \u_\eps^{(1)}\|_{L_2(\R^3)} \le C \eps (\| \q\|_{L_2(\R^3)} + \| \r \|_{L_2(\R^3)}).
$$
Here $\u_\eps^{(1)}$ is interpreted as the zero order corrector; it is expressed in terms of $\u_0$, the solution of some ``correction'' Maxwell system, and some  rapidly oscillating factor. 
The weak limit of  $\u_\eps^{(1)}$ is equal to zero.

\subsection{Statement of the problem. Main results}
In the present paper, we study homogenization of the stationary Maxwell system  in a bounded domain $\O \subset \R^3$ of class  $C^{1,1}$. We rely on the general theory of the Maxwell operator in arbitrary domains developed in the papers \cite{BS1,BS2} by Birman and Solomyak.

Suppose that the magnetic permeability is given by the constant positive matrix $\mu_0$, and the dielectric permittivity is given by the oscillating matrix $\eta^\eps(\x)$. The boundary conditions of perfect conductivity are imposed. The notation for the physical fields is the same as above in Subsection \ref{sec0.2}. The Maxwell operator 
$M_\eps$, written in terms of the displacement vectors, acts in the space $J(\O) \oplus J_0(\O)$. Here 
$J(\O)$ and $J_0(\O)$ are the divergence-free subspaces of  $L_2(\O;\C^3)$ defined below in  
\eqref{5.1},~\eqref{5.2}. The operator $M_\eps$ is given by 
$$
M_\eps = \begin{pmatrix}
0 & i \rot \mu_0^{-1} \cr - i \rot (\eta^\eps)^{-1} & 0
\end{pmatrix}
$$
on the natural domain with  the boundary conditions taken into account (see \eqref{Dom_M_eps}  below). The operator $M_\eps$ is selfadjoint, if  $J(\O)\oplus J_0(\O)$ is considered as a subspace of the weighted space
$$
L_2(\O;\C^3; (\eta^\eps)^{-1}) \oplus L_2(\O;\C^3; \mu_0^{-1}).
$$ 

We study the resolvent $(M_\eps - iI)^{-1}$. In other words, we are interested in the behavior of the solutions $(\w_\eps, \z_\eps)$ of the Maxwell system 
\begin{equation}
\label{0.8a}
(M_\varepsilon - iI) \begin{pmatrix} \w_\eps \cr \z_\eps \end{pmatrix} = \begin{pmatrix} \q \cr \r \end{pmatrix}, 
\quad \q \in J(\O),\ \r \in J_0(\O),
\end{equation}
 and also in the behavior of the fields  $\u_\eps = (\eta^\eps)^{-1}\w_\eps$ and $\v_\eps = \mu_0^{-1}\z_\eps$.

Let $M^0$ be the homogenized Maxwell operator with the coefficients  $\eta^0$ and $\mu_0$.
The homogenized Maxwell system is of the form 
\begin{equation*}
(M^0 - iI) \begin{pmatrix} \w_0 \cr \z_0 \end{pmatrix} = \begin{pmatrix} \q \cr \r \end{pmatrix}.
\end{equation*} 
We put $\u_0 = (\eta^0)^{-1}\w_0$ and $\v_0 = \mu_0^{-1}\z_0$.
As for the problem in  $\R^3$, the classical results  (see \cite{BeLPap,  Sa, ZhKO}) give weak convergence in
$L_2(\O;\C^3)$ of the vector-valued functions $\u_\eps, \w_\eps, \v_\eps, \z_\eps$ to the corresponding homogenized fields $\u_0, \w_0, \v_0, \z_0$.

We improve the classical results in the case where $\q=0$. Let us describe our main results. If $\q=0$,  the fields $\v_\eps$ and $\z_\eps$ converge in the $L_2(\O;\C^3)$-norm to  
$\v_0$ and $\z_0$. The following sharp-order estimates hold:
$$
\begin{aligned}
\| \v_\eps - \v_0 \|_{L_2(\O)} \le C \eps \|\r\|_{L_2(\O)},
\\
\| \z_\eps - \z_0 \|_{L_2(\O)} \le C \eps \|\r\|_{L_2(\O)}.
\end{aligned}
$$
 In addition, we find approximations for   $\v_\eps$ and $\z_\eps$ in the $H^1(\O;\C^3)$-norm:
$$
\begin{aligned}
\big\| \v_\eps - \v_0 - \eps \v_\eps^{(1)}\big\|_{H^1(\O)} \le C \eps^{1/2} \|\r\|_{L_2(\O)},
\\
\big\| \z_\eps - \z_0 - \eps \z_\eps^{(1)} \big\|_{H^1(\O)} \le C \eps^{1/2} \|\r\|_{L_2(\O)}.
\end{aligned}
$$
Here the correctors $\v_\eps^{(1)}$ and $\z_\eps^{(1)}$ involve rapidly oscillating factors, their norms in  $H^1(\O;\C^3)$ are of order $O(\eps^{-1})$. 
Finally, we obtain approximations for  $\u_\eps$ and $\w_\eps$ in the $L_2(\O;\C^3)$-norm:
$$
\begin{aligned}
\big\| \u_\eps - \u_0 - \u_\eps^{(1)}\big\|_{L_2(\O)} \le C \eps^{1/2} \|\r\|_{L_2(\O)},
\\
\big\| \w_\eps - \w_0 - \w_\eps^{(1)} \big\|_{L_2(\O)} \le C \eps^{1/2} \|\r\|_{L_2(\O)}.
\end{aligned}
$$
The correction terms $\u_\eps^{(1)}$ and $\w_\eps^{(1)}$ can be interpreted as correctors of zero order, they weakly tend to zero. 

The case of system \eqref{0.8a} with $\r=0$ is more difficult and is not considered in the present paper. 

\subsection{The method} As for the problem in  $\R^3$, the method is based on reduction to the study of some auxiliary second order operator $L_\eps$.
First, we study this operator, and next we derive the results for the Maxwell system. 

The operator $L_\eps$ acts in $L_2(\O;\C^3)$ and is formally given by 
\begin{equation}
\label{0.12}
L_\eps = \mu_0^{-1/2} \rot (\eta^\eps(\x))^{-1} \rot \mu_0^{-1/2} - \mu_0^{1/2} \nabla \nu^\eps(\x)\div \mu_0^{1/2} 
\end{equation}
with the boundary conditions 
\begin{equation}
\label{0.13}
(\mu_0^{1/2} \f)_n\vert_{\partial \O} =0, \quad  ((\eta^\eps(\x))^{-1} \rot (\mu_0^{-1/2} \f))_\tau \vert_{\partial \O} =0.
\end{equation}
The precise definition of the operator $L_\eps$ is given in terms of the quadratic form.
For application to the Maxwell system, we can put ${\nu(\x)=1}$, but for  generality we study the operator \eqref{0.12} with variable coefficient  $\nu^\eps(\x)$.
The operator $L_\eps$ can be written in a factorized form $b(\D)^* g^\eps(\x)b(\D)$, but
the direct reference to the results of \cite{PSu, Su13, Su_SIAM} is impossible, since in those papers
the cases of the Dirichlet or Neumann boundary conditions were studied, and in the present case 
 the boundary conditions \eqref{0.13} are of mixed type. 
Therefore, we need to prove analogs of estimates \eqref{0.4} and \eqref{0.5} for the resolvent of the operator $L_\eps$. 

The method of proving such estimates is based on consideration of the associated problem in $\R^3$ and
using the results for this problem, introduction of the boundary layer correction term $\s_\eps$, and a careful analysis of this term. A crucial role is played by  using the 
Steklov smoothing operator (initially borrowed from \cite{Zh1, ZhPas1}), estimates in the  $\eps$-neighborhood of the boundary, and the duality arguments.

\subsection{The plan of the paper} The paper consists of five sections. In~Section \ref{Sec1},  
the model second order operator $\L_\eps$ in  $L_2(\R^3;\C^3)$ is considered;
the effective operator is constructed, and the known results about approximation of the resolvent \hbox{$(\L_\eps+I)^{-1}$}
are formulated.
In Section \ref{Sec2}, the model operator $L_\eps$ in $L_2(\O;\C^3)$ is introduced, the effective operator is described, and  some auxiliary statements  (about estimates in the $\eps$-neighborhood of the boundary) are given. 
In Section \ref{Sec3}, we formulate main results about approximation of the resolvent  ${(L_\eps +I)^{-1}}$ and 
give the first two steps of the proofs: the associated problem in  $\R^3$ is considered,
the boundary layer correction term $\s_\eps$ is introduced, and the proof of main theorems is reduced to  estimation of $\s_\eps$ in $H^1(\O;\C^3)$ and in $L_2(\O;\C^3)$.
In Section \ref{Sec4}, we obtain the required estimates for the norms of the correction term $\s_\eps$ and complete 
the proof of theorems from  Section 3.
Section \ref{Sec5} is devoted to homogenization of the stationary Maxwell system with $\q=0$. 
We reduce the problem to the question about the behavior of the resolvent of $L_\eps$.
The final result on  approximation for the solutions of the Maxwell system (Theorem \ref{th_Maxwell}) is obtained.

\subsection{Notation} 
Let $\mathfrak{H}$ and $\mathfrak{H}_*$ be complex separable Hilbert spaces. The symbols $(\,\cdot\, ,\,\cdot\,)_\mathfrak{H}$ and $\Vert \,\cdot\,\Vert _\mathfrak{H}$ stand for the inner product and the norm in $\mathfrak{H}$; the symbol $\Vert \,\cdot\,\Vert _{\mathfrak{H}\rightarrow\mathfrak{H}_*}$ denotes the norm of a linear continuous operator acting from $\mathfrak{H}$ to $\mathfrak{H}_*$.

The symbols $\langle \,\cdot\, ,\,\cdot\,\rangle$ and $\vert \,\cdot\,\vert$ stand for the inner product and the norm in  $\mathbb{C}^n$; $\mathbf{1}= \mathbf{1}_n$ is the identity $(n\times n)$-matrix. If $a$ is an $(n\times n)$-matrix, then $\vert a\vert$ denotes the norm of $a$ as a linear operator in $\mathbb{C}^n$. 
We denote $\mathbf{x}=(x_1,x_2, x_3)\in\mathbb{R}^3$, $iD_j=\partial _j =\partial /\partial x_j$, $j=1,2,3$, $\mathbf{D}=-i\nabla=(D_1,D_2,D_3)$. The classes $L_p$ of $\mathbb{C}^n$-valued functions in a domain $\mathcal{O}\subset\mathbb{R}^3$ are denoted by $L_p(\mathcal{O};\mathbb{C}^n)$, $1\leqslant p\leqslant \infty$. The Sobolev spaces of $\mathbb{C}^n$-valued functions in a domain  $\mathcal{O}$ are denoted by $H^s(\mathcal{O};\mathbb{C}^n)$. Next, $H^1_0(\mathcal{O};\mathbb{C}^n)$ is the closure of  $C_0^\infty (\mathcal{O};\mathbb{C}^n)$ in  $H^1(\mathcal{O};\mathbb{C}^n)$. If $n=1$, we write simply $L_p(\mathcal{O})$, $H^s(\mathcal{O})$, etc., but sometimes we use such simple notation for the spaces of vector-valued or matrix-valued functions. 
Various constants in estimates are denoted by $c$, $\mathfrak c$, $C$, $\mathcal{C}$, $\mathfrak{C}$ 
(possibly, with indices and marks).

\subsection{}
The author plans to devote a separate paper to more general problem 
about homogenization of the stationary Maxwell system in a bounded domain in the case where both coefficients are given by the rapidly oscillating periodic matrix-valued functions.
Problem \eqref{0.8a} with $\r=0$ (which is not considered in the present paper) will be a particular case of this more general problem. 

The author is grateful to N.~D.~Filonov for  consultation concerning the properties of the Maxwell operator and useful comments.

\section{The model second order operator in $\R^3$\label{Sec1}}

\subsection{Lattice\label{Sec1.1}} Let $\Gamma\subset\mathbb{R}^3$ be a lattice generated by the basis 
 $\a_1, \a_2, \a_3$, i.~e., 
$$
\Gamma = \big\{ \a \in \R^3:\ \a = z_1 \a_1 + z_2 \a_2 + z_3 \a_3,\ z_j \in \Z\big\}. 
$$
By $\Omega \subset \R^3$ we denote the elementary cell of the lattice $\Gamma$:
$$
\Omega = \big\{ \x \in \R^3:\ \x = t_1 \a_1 + t_2 \a_2 + t_3 \a_3,\ -1/2 < t_j < 1/2 \big\}. 
$$
For $\Gamma$-periodic functions $f(\x)$ in $\mathbb{R}^3$, we use the notation $f^\varepsilon (\mathbf{x}):=f(\mathbf{x}/\varepsilon)$, where $\varepsilon >0$.
For periodic square matrix-valued functions $f(\x)$, we denote
$$
\overline{f}:=\vert \Omega\vert ^{-1}\int\limits_\Omega f(\mathbf{x})\,d\mathbf{x}, \quad \underline{f}:=\Big(\vert \Omega\vert ^{-1}\int\limits_\Omega f(\mathbf{x})^{-1}\,d\mathbf{x}
\Big)^{-1}.
$$
Here, in the definition of $\overline{f}$ it is assumed that  $f \in L_{1,\text{loc}}(\R^3)$, and in the definition of  $\underline{f}$ it is assumed that the matrix $f(\x)$ is non-degenerate and  
$f^{-1} \in L_{1,\text{loc}}(\R^3)$.

By $\widetilde{H}^1(\Omega;\C^n)$ we denote the subspace of functions in $H^1(\Omega;\C^n)$, whose $\Gamma$-periodic extension  to  $\mathbb{R}^3$ belongs to  $H^1_{\mathrm{loc}}(\mathbb{R}^3;\C^n)$.

\subsection{The Steklov smoothing\label{Sec1.2a}} 
 The operator $S_\varepsilon^{(k)}\!,$ ${\varepsilon \!>\!0}$, acting in $L_2(\mathbb{R}^3;\mathbb{C}^k)$ (where $k\in\mathbb{N}$) and given by  
\begin{equation}
\label{S_eps}
\begin{split}
(S_\varepsilon^{(k)} \mathbf{u})(\mathbf{x})=\vert \Omega \vert ^{-1}\int\limits_\Omega \mathbf{u}(\mathbf{x}-\varepsilon \mathbf{z})\,d\mathbf{z},\quad \mathbf{u}\in L_2(\mathbb{R}^3;\mathbb{C}^k),
\end{split}
\end{equation}
is called the  \textit{Steklov smoothing operator}.
We omit the index $k$ and write simply $S_\varepsilon$. Obviously,
$S_\varepsilon \mathbf{D}^\alpha \mathbf{u}=\mathbf{D}^\alpha S_\varepsilon \mathbf{u}$ for  $\mathbf{u}\in H^\sigma(\mathbb{R}^3;\mathbb{C}^k)$ and any multiindex $\alpha$ such that $\vert \alpha\vert \leqslant \sigma$.
Note that  
\begin{equation}
\label{S_eps <= 1}
\Vert S_\varepsilon \Vert _{L_2(\mathbb{R}^3)\rightarrow L_2(\mathbb{R}^3)}\leqslant 1.
\end{equation}
We need the following properties of the operator $S_\varepsilon$
(see \cite[Lemmas~1.1 and~1.2]{ZhPas1} or \cite[Propositions 3.1 and 3.2]{PSu}).

\begin{proposition}
\label{prop_Seps - I}
For any function  $\mathbf{u}\in H^1(\mathbb{R}^3;\mathbb{C}^k)$, we have  
\begin{equation*}
\Vert S_\varepsilon \mathbf{u}-\mathbf{u}\Vert _{L_2(\mathbb{R}^3)}\leqslant \varepsilon r_1\Vert \mathbf{D}\mathbf{u}\Vert _{L_2(\mathbb{R}^3)},
\end{equation*}
where $2r_1=\mathrm{diam}\,\Omega$.
\end{proposition}

\begin{proposition}
\label{prop f^eps S_eps}
Let $f$ be a $\Gamma$-periodic function in $\mathbb{R}^3$ such that $f \in L_2(\Omega)$. 
Let $[f ^\varepsilon ]$ be the operator of multiplication by the function $f^\eps(\x)$. 
Then the operator $[f ^\varepsilon ]S_\varepsilon $ is continuous in $L_2(\mathbb{R}^3)$ and 
\begin{equation*}
\Vert [f^\varepsilon]S_\varepsilon \Vert _{L_2(\mathbb{R}^3)\rightarrow L_2(\mathbb{R}^3)}\leqslant \vert \Omega \vert ^{-1/2}\Vert f \Vert _{L_2(\Omega)}.
\end{equation*}
\end{proposition}

\subsection{Definition of the operator $\mathcal{L}_\eps$\label{Sec1.2}}
Suppose that $\mu_0$ is a symmetric  positive $(3\times 3)$-matrix with real entries.
Suppose that a symmetric $(3\times 3)$-matrix-valued function $\eta(\x)$ with real entries  and a real-valued function $\nu(\x)$ are periodic with respect to the lattice $\Gamma$ and such that 
\begin{equation}\label{cond_eta_nu}
\eta, \eta^{-1} \in L_\infty,\quad \eta(\x)>0;\quad 
\nu, \nu^{-1} \in L_\infty,\quad \nu(\x)>0.
\end{equation}

In $L_2(\R^3;\mathbb{C}^3)$, we consider the operator $\L_\eps$ given formally by the differential expression 
\begin{equation}\label{oper_L_eps}
{\L}_\eps = \mu_0^{-1/2} \rot (\eta^\eps(\x))^{-1} \rot \mu_0^{-1/2} - \mu_0^{1/2} \nabla \nu^\eps(\x) \div \mu_0^{1/2}.
\end{equation}
The operator $\L_\eps$ belongs to the class of operators admitting a factorization of the form  \eqref{0.1}, i.~e.,
$\L_\eps = b(\D)^* g^\eps(\x) b(\D)$. This class was studied in the papers \cite{BSu1,BSu2,BSu3}.
In our case,  $g(\x)$ is the $(4\times 4)$-matrix-valued function, and $b(\D)$ is the $(4\times 3)$-matrix first order differential operator. Namely, 
\begin{equation}\label{def}
b(\D) = \begin{pmatrix} -i \rot \mu_0^{-1/2} \cr -i \div \mu_0^{1/2}\end{pmatrix},
\quad 
g(\x) = \begin{pmatrix} \eta(\x)^{-1} & 0 \cr 0 & \nu(\x) \end{pmatrix}.
\end{equation}
From \eqref{cond_eta_nu} it follows that the matrix $g(\x)$ is positive definite and bounded. 
Obviously,
$$
\|g\|_{L_\infty} = \max \big\{ \|\eta^{-1}\|_{L_\infty}; \|\nu \|_{L_\infty}\big\},\quad 
\|g^{-1}\|_{L_\infty} = \max \big\{ \|\eta\|_{L_\infty}; \|\nu^{-1} \|_{L_\infty}\big\}.
$$
The operator $b(\D)$ can be written as  $b(\D)= \sum_{j=1}^3 b_j D_j$, where $b_j$ are constant matrices. 
The symbol $b(\bxi)= \sum_{j=1}^3 b_j \xi_j$ of the operator $b(\D)$ is given by  
$$
b(\bxi) = \begin{pmatrix} r(\bxi) \mu_0^{-1/2} \cr \bxi^t \mu_0^{1/2} \end{pmatrix},
\quad r(\bxi) = \begin{pmatrix} 0 & -\xi_3 & \xi_2 \cr \xi_3 & 0 & -\xi_1 \cr -\xi_2 & \xi_1 & 0 \end{pmatrix},
\quad \bxi^t = \begin{pmatrix} \xi_1 & \xi_2 & \xi_3 \end{pmatrix}.
$$
We have
\begin{equation}\label{rank_cond}
\rank b(\bxi) =3, \quad 0 \ne \bxi \in \R^3. 
\end{equation}
This condition is equivalent to the estimates 
\begin{equation}\label{bb_eps}
\alpha_0 \1_3 \le b(\bxi)^* b(\bxi) \le  \alpha_1 \1_3, \quad |\bxi|=1,
\end{equation}
with positive constants $\alpha_0$ and $\alpha_1$.
It is easy to check these estimates with the constants 
$$
\alpha_0 = \min \big\{ |\mu_0|^{-1}; |\mu_0^{-1}|^{-1} \big\}, \quad \alpha_1 = |\mu_0| + |\mu_0^{-1}|. 
$$

The precise definition of the operator $\L_\eps$ is given in terms of the quadratic form
\begin{equation*}
\begin{split}
&{\l}_\eps[\f,\f] := \intop_{\R^3} \langle g^\eps(\x) b(\D) \f, b(\D) \f \rangle \, d\x
\\
&=\intop_{\R^3} \left( \langle (\eta^\eps(\x))^{-1} \rot (\mu_0^{-1/2}\f), \rot (\mu_0^{-1/2}\f)\rangle + \nu^\eps(\x) |\div (\mu_0^{1/2}\f)|^2\right) \, d\x, 
\quad\f \in H^1(\R^3;\C^3).
\end{split}
\end{equation*}
Under our assumptions, the following two-sided estimates hold: 
\begin{equation}\label{form_l_eps_est}
\begin{split}
c_1 \| \D \f \|^2_{L_2(\R^3)} \le {\l}_\eps[\f,\f] \le c_2 \| \D \f \|^2_{L_2(\R^3)}, \quad \f \in H^1(\R^3;\C^3),
\\
c_1 = \alpha_0 \|g^{-1}\|^{-1}_{L_\infty}, \quad  c_2 = \alpha_1 \|g\|_{L_\infty}.
\end{split}
\end{equation}
Thus, the form $\l_\eps$ is closed and nonnegative. The selfadjoint operator in  $L_2(\R^3;\C^3)$ generated by this form is denoted by $\L_\eps$.

\subsection{The effective operator $\L^0$\label{Sec1.2a}} 
According to the general rules,  we define the \textit{effective operator} 
\begin{equation}\label{L0}
\L^0 = b(\D)^* g^0 b(\D),
\end{equation}
where $g^0$ is a constant positive matrix called the  \textit{effective matrix}. It is defined in terms of 
the solution of the auxiliary problem on the cell $\Omega$.
Let $\Lambda(\x)$ be a $(3\times 4)$-matrix-valued function which is a  $\Gamma$-periodic solution of the problem
\begin{equation}\label{Lambda_eq}
b(\D)^* g(\x) (b(\D) \Lambda(\x) + \1) =0, \quad \intop_\Omega \Lambda(\x) \, d\x =0.
\end{equation}
Then  
\begin{equation}\label{eff_matrix}
g^0 = |\Omega|^{-1} \intop_\Omega \wt{g}(\x) \, d\x, \quad 
\wt{g}(\x) := g(\x) (b(\D) \Lambda(\x) + \1).
\end{equation}
It is easy to check that
\begin{equation}\label{Lambda_est}
\| \Lambda \|_{H^1(\Omega)} \le {\mathfrak C}_\Lambda |\Omega|^{1/2}, 
\end{equation}
where the constant ${\mathfrak C}_\Lambda$ depends only on $|\mu_0|$, $|\mu_0^{-1}|$, $\|\eta\|_{L_\infty}$, $\|\eta^{-1}\|_{L_\infty}$, 
$\|\nu\|_{L_\infty}$, $\|\nu^{-1}\|_{L_\infty}$, and the parameters of the lattice $\Gamma$.

The effective operator for $\L_\eps$ was constructed in  \cite[Chapter 7]{BSu1} in the case where $\mu_0=\1$ 
and in \cite{BSu4} in the general case. For completeness, we repeat the corresponding constructions. 

Let us find the matrix $\Lambda(\x)$. Let $\e_j$, $j=1,2,3,4,$ be the standard orthonormal basis in $\C^4$ and let  
$\wt{\e}_j$, $j=1,2,3,$ be the standard orthonormal basis in  $\C^3$. 
Let  $\CC = \sum_{j=1}^4 C_j \e_j \in \C^4$. The vector
$\wt{\CC} = \sum_{j=1}^3 C_j \wt{\e}_j \in \C^3$ corresponds to $\CC$.
The vector-valued function $\v = \Lambda \CC \in \wt{H}^1(\Omega;\C^3)$ is the solution of the equation  
$b(\D)^* g(\x) (b(\D) \v(\x) + \CC) =0$ which now takes the form
$$
 \mu_0^{-1/2} \rot (\eta(\x))^{-1} \left( \rot (\mu_0^{-1/2} \v) + i \wt{\CC} \right)  
- \mu_0^{1/2} \nabla \nu(\x) \left( \div (\mu_0^{1/2} \v) + i C_4 \right)=0.
$$
In other words, $\v  \in \wt{H}^1(\Omega;\C^3)$ satisfies the identity
\begin{equation}\label{v_eq}
\begin{split}
&\intop_{\Omega} \big\langle (\eta(\x))^{-1} (\rot (\mu_0^{-1/2}\v) + i \wt{\CC}), \rot (\mu_0^{-1/2}\z) \big\rangle \,d\x
\\
   &+\intop_{\Omega}  \nu(\x) \left(\div (\mu_0^{1/2}\v) + i C_4 \right) \overline{\div (\mu_0^{1/2}\z)}     \, d\x =0, 
\quad \z \in \wt{H}^1(\Omega;\C^3).
\end{split}
\end{equation}
Using decomposition $\mu_0^{-1/2} \z = {\mathbf f} + \nabla h$, where 
${\mathbf f} \in \wt{H}^1(\Omega;\C^3)$ and $\div (\mu_0 \mathbf{f})=0$
(the Weyl decomposition), we write identity \eqref{v_eq} with $\z = \mu_0^{1/2} {\mathbf f}$. 
Then the second term in  \eqref{v_eq} is equal to zero. Since $\rot \mathbf{f} = \rot (\mu_0^{-1/2} \z)$, we arrive at 
\begin{equation}\label{v_eq2}
\intop_{\Omega} \big\langle (\eta(\x))^{-1}\big (\rot (\mu_0^{-1/2}\v) + i \wt{\CC}\big),\  \rot (\mu_0^{-1/2}\z) \big\rangle \,d\x,
\quad \z \in \wt{H}^1(\Omega;\C^3).
\end{equation}
From \eqref{v_eq2} it follows that
\begin{equation*}
(\eta(\x))^{-1} \big(\rot (\mu_0^{-1/2}\v(\x)) + i \wt{\CC}\big)  =
i \big(\nabla \Phi(\x) + \c\big)
\end{equation*}
with some $\Phi \in \wt{H}^1(\Omega)$ and $\c \in \C^3$. Hence,
\begin{equation}\label{v_eq4}
\rot \big(\mu_0^{-1/2}\v(\x)\big) + i \wt{\CC}  =
i \eta(\x)\big(\nabla \Phi(\x) + \c\big). 
\end{equation}
By \eqref{v_eq4}, 
\begin{equation*}
\intop_{\Omega} \big\langle 
\eta(\x)\big(\nabla \Phi(\x) + \c\big) , \nabla F(\x) \big\rangle \, d\x =0,
\quad F \in \wt{H}^1(\Omega).
\end{equation*}
Thus, $\Phi \in \wt{H}^1(\Omega)$ is the solution of the equation
\begin{equation}\label{v_eq6}
\div \eta(\x)(\nabla \Phi(\x) + \c)  =0. 
\end{equation}
Recalling the definition of the effective matrix $\eta^0$ for the operator $-\div \eta(\x) \nabla$, we have
\begin{equation}\label{eta0}
\eta^0 \c = |\Omega|^{-1}\intop_{\Omega}  \eta(\x)(\nabla \Phi(\x) + \c)\, d\x. 
\end{equation}
Integrating \eqref{v_eq4} and using  \eqref{eta0}, we find 
\begin{equation}\label{relation}
\wt{\CC}= \eta^0 \c.
\end{equation}

On the other hand, \eqref{v_eq} and \eqref{v_eq2} imply that 
\begin{equation*}
 \intop_{\Omega}  \nu(\x) \Big(\div (\mu_0^{1/2}\v) + i C_4 \Big) \overline{\div (\mu_0^{1/2}\z)}     \, d\x =0, 
\quad \z \in \wt{H}^1(\Omega;\C^3).
\end{equation*}
This means that there exists a constant  $\alpha \in \C$ such that  
\begin{equation}\label{v_eq8}
\nu(\x) \left(\div (\mu_0^{1/2}\v) + i C_4 \right) = i \alpha.
\end{equation}
Multiplying  \eqref{v_eq8} by $\nu(\x)^{-1}$ and integrating, we obtain  
$C_4 |\Omega|  = \alpha \int_\Omega \nu(\x)^{-1}\, d\x.$ Hence, 
\begin{equation}\label{v_eq8a}
\alpha = \underline{\nu} C_4.
\end{equation}

Substituting $\CC = \e_j$, we find the columns  $\v_j(\x) = \Lambda(\x)\e_j$ of the matrix $\Lambda(\x)$. 
From \eqref{v_eq4},  \eqref{relation}, \eqref{v_eq8}, and \eqref{v_eq8a}  it follows that 
for  $j=1,2,3$ the column $\v_j \in \wt{H}^1(\Omega;\C^3)$ is a periodic solution of the problem
\begin{equation}\label{v_eq9}
\begin{split}
\rot (\mu_0^{-1/2}\v_j(\x))   = i \eta(\x)(\nabla \Phi_j(\x) + \c_j) -  i \wt{\e}_j,
\\
\div (\mu_0^{1/2}\v_j(\x))=0, \quad \intop_\Omega \v_j(\x)\,d\x =0.
\end{split}
\end{equation}
Here $\c_j = (\eta^0)^{-1} \wt{\e}_j$, and $\Phi_j \in \wt{H}^1(\Omega)$ is the solution of equation 
 \eqref{v_eq6} with $\c = \c_j$.
The solution of problem \eqref{v_eq9} can be represented as 
\begin{equation*}
\v_j(\x) = i \mu_0^{-1/2} \rot \p_j(\x), 
\end{equation*}
where $\p_j \in \wt{H}^1(\Omega;\C^3)$ is the solution of the problem  
\begin{equation}\label{v_eq11}
\begin{split}
\rot (\mu_0^{-1} \rot \p_j(\x))  & =  \eta(\x)(\nabla \Phi_j(\x) + \c_j) -  \wt{\e}_j,
\\
\div \p_j (\x) &=0, \quad \intop_\Omega \p_j(\x) \, d\x =0.
\end{split}
\end{equation}

For $\CC = \e_4$, relations \eqref{v_eq4},  \eqref{relation}, \eqref{v_eq8}, and \eqref{v_eq8a} show that 
 $\v_4 \in \wt{H}^1(\Omega;\C^3)$ is the periodic solution of the problem  
\begin{equation*}
\begin{split}
\rot (\mu_0^{-1/2}\v_4(\x))   = 0,
\quad
\div (\mu_0^{1/2}\v_4(\x))= i \left( \underline{\nu} \nu(\x)^{-1} - 1 \right), \quad \intop_\Omega \v_4(\x)\, d\x=0.
\end{split}
\end{equation*}
Hence, $\v_4(\x)= i \mu_0^{1/2} \nabla \rho(\x)$, where  $\rho(\x)$ is the $\Gamma$-periodic solution of the equation
\begin{equation}\label{v_eq13}
- \div (\mu_0 \nabla \rho (\x))= 1 -  \underline{\nu} \nu(\x)^{-1}.
\end{equation}
 
Thus, the matrix $\Lambda(\x)$ takes the form 
\begin{equation}\label{v_eq13a}
 \Lambda (\x) = i \begin{pmatrix} \mu_0^{-1/2} \Psi(\x) & \mu_0^{1/2} \nabla \rho(\x) \end{pmatrix},
\end{equation}
where $\Psi(\x)$ is the $(3\times 3)$-matrix with the columns $\rot \p_j(\x)$, $j=1,2,3$. 
The matrix $\wt{g}(\x)= g(\x)(b(\D) \Lambda(\x) + \1)$ is represented as 
$$
\wt{g}(\x)= \begin{pmatrix} \eta(\x)^{-1} (\rot (\mu_0^{-1}\Psi(\x)) + \1_3) & 0 \cr 0 & \nu(\x) (\div (\mu_0 \nabla \rho(\x)) +1) \end{pmatrix}.
$$
By \eqref{v_eq11} and \eqref{v_eq13}, 
\begin{equation}\label{v_eq14}
 \wt{g}(\x) =  \begin{pmatrix} \Sigma(\x) + (\eta^0)^{-1}   & 0 \cr 0 &  \underline{\nu} \end{pmatrix},
\end{equation}
where $\Sigma(\x)$ is the $(3 \times 3)$-matrix with the columns $\nabla \Phi_j(\x)$, $j=1,2,3$. 
Together with  \eqref{eff_matrix}, this implies
\begin{equation}\label{v_eq15}
 {g}^0  =  \begin{pmatrix}  (\eta^0)^{-1}   & 0 \cr 0 &  \underline{\nu} \end{pmatrix}.
\end{equation}
Consequently, the effective operator \eqref{L0} is given by the differential expression 
\begin{equation}\label{oper_effective}
{\L}^0 = \mu_0^{-1/2} \rot (\eta^0)^{-1} \rot \mu_0^{-1/2} - \mu_0^{1/2} \nabla \underline{\nu} \div \mu_0^{1/2}
\end{equation}
on the domain $H^2(\R^3;\C^3)$.

The following estimates for the effective coefficients are well known:
\begin{equation}\label{eff_coeff_est}
\begin{split}
|\eta^0| &\le \| \eta \|_{L_\infty},\quad | (\eta^0)^{-1}| \le \| \eta^{-1} \|_{L_\infty},
\\ |\underline{\nu} | &\le \| \nu \|_{L_\infty},\ \quad |(\underline{\nu})^{-1}| \le \| \nu^{-1} \|_{L_\infty}.
\end{split}
\end{equation}
The symbol of the effective operator is given by 
$$
a(\bxi) = \mu_0^{-1/2} r(\bxi)^t (\eta^0)^{-1} r(\bxi) \mu_0^{-1/2} + \mu_0^{1/2} \bxi \underline{\nu} \bxi^t  \mu_0^{1/2}.
$$
Taking \eqref{eff_coeff_est} into account, we see that the symbol $a(\bxi)$ satisfies 
\begin{equation}\label{symbol_eff_est}
c_1 |\bxi|^2 \1 \le  a(\bxi) \le c_2 |\bxi|^2 \1, \quad \bxi \in \R^3.
\end{equation}
Here the constants  $c_1$ and $c_2$ are the same as in  \eqref{form_l_eps_est}.

\subsection{The properties of the effective matrix $\eta^0$. 
The properties of the functions $\Phi_j$\label{Sec1.2a}}

The effective matrix $\eta^0$ satisfies the estimates 
\begin{equation}\label{FR}
\underline{\eta} \le \eta^0 \le \overline{\eta},
\end{equation}
known as the Voigt--Reuss bracketing. See, e.~g., \cite[Chapter~3, Theorem~1.5]{BSu1}.
We distinguish the cases where one of the inequalities in \eqref{FR} becomes an identity; see, e.~g., \cite[Chapter~3, Propositions~1.6 and~1.7]{BSu1}.

\begin{proposition}\label{prop_cases}

\noindent\emph{1)} The identity $\eta^0 = \overline{\eta}$ is equivalent to the relations
$\div {\boldsymbol{\eta}}_j(\x)=0$, $j=1,2,3,$ for the columns 
${\boldsymbol{\eta}}_j(\x)$ of the matrix  $\eta(\x)$.

\noindent\emph{2)} The identity $\eta^0 = \underline{\eta}$ is equivalent to the following representations 
for the columns ${\boldsymbol{\kappa}}_j(\x)$ of the matrix $\eta(\x)^{-1}$\emph{:}
${\boldsymbol{\kappa}}_j(\x)= {\mathbf c}_j^0 + \nabla f_j(\x)$, $j=1,2,3$, with some ${\mathbf c}_j^0 \in \C^3$ and $f_j \in \wt{H}^1(\Omega)$.
\end{proposition}

\begin{remark}\label{rem_eta0}
1) If $\eta^0 = \overline{\eta}$, then $\Phi_j(\x)=0$, $j=1,2,3,$ and $\Sigma(\x)=0$. 
According to \eqref{v_eq14}, in this case we have $\wt{g}(\x) = g^0$.

\noindent 2) If $\eta^0 = \underline{\eta}$, then $\eta(\x)(\nabla \Phi_j(\x) + {\mathbf c}_j)= \wt{\e}_j$, $j=1,2,3;$
see \cite[Remark 3.5]{BSu2}. 
In this case, we have $\v_j(\x) = 0$, $j=1,2,3,$ i.~e., $\Psi(\x)=0$. If, in addition, $\nu(\x)= \operatorname{Const}$, then $\v_4(\x)=0$. Hence, in this case we have  $\Lambda(\x)=0$.
\end{remark}

In what follows, we will need some properties of the functions $\Phi_j$, $j=1,2,3$. 

\begin{remark}\label{LaUr}
The columns of the matrix $\Sigma(\x)$  are vector-functions $\nabla \Phi_j(\x),$ $j=1,2,3$, 
where $\Phi_j$ is the periodic solution of the problem 
\begin{equation}\label{1.34a}
\div \eta(\x) (\nabla \Phi_j(\x) + \c_j) =0, \quad \intop_\Omega \Phi_j(\x) \,d\x=0,
\end{equation}
with $\c_j = (\eta^0)^{-1} \wt{\e}_j$.
According to \cite[Chapter~3, Theorem~13.1]{LaUr}, the solution of this problem is bounded: $\Phi_j \in L_\infty$, and the norm 
$\|\Phi_j\|_{L_\infty}$ is controlled in terms of $\|\eta\|_{L_\infty}$, $\|\eta^{-1}\|_{L_\infty}$, and the parameters of the lattice $\Gamma$. 
\end{remark}

The following statement was checked in  \cite[Corollary 2.4]{PSu}.

\begin{proposition}\label{prop_PSu}
 For any  $u \in H^1(\R^3)$, we have  
\begin{equation*}
\intop_{\R^3} |(\nabla \Phi_j)^\eps|^2 |u|^2 \, d\x \le 
\beta_1 \|u\|^2_{L_2(\R^3)} + \beta_2 \eps^2 \|\Phi_j\|_{L_\infty}^2  \| \D u\|^2_{L_2(\R^3)},
\end{equation*}
where the constants $\beta_1$ and $\beta_2$ depend only on $\|\eta\|_{L_\infty}$ and $\|\eta^{-1}\|_{L_\infty}$. 
\end{proposition}

\subsection{Approximation of the resolvent of the operator $\mathcal{L}_\eps$\label{Sec1.3}}
Applying Theorem 2.1 from \cite[Chapter~4]{BSu1} to the operator~\eqref{oper_L_eps}, we obtain 
the following result.

\begin{theorem}\label{th1}
Let $\L_\eps$ be the operator~\eqref{oper_L_eps}. Suppose that the effective operator $\L^0$ is defined by \eqref{oper_effective}. For $\eps>0$ we have  
$$
\big\| ({\L}_\eps+I)^{-1} - (\L^0 +I)^{-1}\big\|_{L_2(\R^3) \to L_2(\R^3)} \le C_1 \eps.
$$
The constant $C_1$ depends only on $|\mu_0|$, $|\mu_0^{-1}|$, $\|\eta\|_{L_\infty}$, $\|\eta^{-1}\|_{L_\infty}$,
$\|\nu \|_{L_\infty}$, $\|\nu^{-1}\|_{L_\infty}$, and the parameters of the lattice~$\Gamma$. 
\end{theorem} 

Approximation of the resolvent in the norm of operators acting from $L_2(\R^3;\C^3)$ to the Sobolev space $H^1(\R^3;\C^3)$ was obtained in  \cite[Theorem 10.6]{BSu3}; 
that approximation contained a corrector with the smoothing operator of different type than $S_\eps$. In \cite[Theorem~3.3]{PSu}, it was shown that this smoothing operator can be replaced by the Steklov smoothing operator. Let us formulate the result of \cite{PSu} as applied to  
\eqref{oper_L_eps}. We introduce a \textit{corrector} 
\begin{equation}\label{corr0}
\K_\eps = \Lambda^\eps S_\eps b(\D) (\L^0 + I)^{-1}. 
\end{equation}  
Here $S_\eps$ is the Steklov smoothing operator defined by  \eqref{S_eps}, 
and the matrix $\Lambda$ is the periodic solution of  problem  \eqref{Lambda_eq}.
The operator  
$$
b(\D) (\L^0 + I)^{-1}
$$
 is continuous from $L_2(\R^3;\C^3)$ to $H^1(\R^3;\C^4)$. By Proposition~\ref{prop f^eps S_eps}
and  relation $\Lambda \in \wt{H}^1(\Omega)$, the operator  $\Lambda^\eps S_\eps$ is a continuous mapping of $H^1(\R^3;\C^4)$ to $H^1(\R^3;\C^3)$. 
Hence, the corrector \eqref{corr0} is continuous from $L_2(\R^3;\C^3)$ to $H^1(\R^3;\C^3)$.  
 Taking  \eqref{def} and \eqref{v_eq13a} into account, we obtain 
\begin{equation}\label{corr1}
\K_\eps = \left( \mu_0^{-1/2}\Psi^\eps S_\eps \rot \mu_0^{-1/2} + \mu_0^{1/2} (\nabla \rho)^\eps S_\eps \div \mu_0^{1/2}
\right) (\L^0 +I)^{-1}. 
\end{equation}

\begin{theorem}\label{th2} 
Suppose that the assumptions of Theorem \emph{\ref{th1}} are satisfied. 
Suppose that the corrector~$\K_\eps$ is defined by  {\eqref{corr1}}. 
Then for $\eps>0$ we have 
$$
\| ( {\L}_\eps+I)^{-1} - (\L^0 +I)^{-1} - \eps \K_\eps\|_{L_2(\R^3) \to H^1(\R^3)} \le C_2 \eps.
$$
The constant $C_2$ depends only on $|\mu_0|$, $|\mu_0^{-1}|$, $\|\eta\|_{L_\infty}$, $\|\eta^{-1}\|_{L_\infty}$,
$\|\nu \|_{L_\infty}$, $\|\nu^{-1}\|_{L_\infty}$, and the parameters of the lattice $\Gamma$. 
\end{theorem} 

It is easy to deduce approximation for the  ``flux'' 
$
g^\eps b(\D) ( {\L}_\eps+I)^{-1}
$
from Theorem~\ref{th2}; see \cite[Theorem 1.8]{Su_SIAM}. We have
\begin{equation}
\label{flux1}
\big\| g^\eps b(\D) ( {\L}_\eps+I)^{-1} - \wt{g}^\eps S_\eps b(\D)( {\L}^0 +I)^{-1} \big\|_{L_2(\R^3) \to L_2(\R^3)} 
\le \check{C}_3 \eps,
\end{equation}
where $\check{C}_3$ depends only on $|\mu_0|$, $|\mu_0^{-1}|$, $\|\eta\|_{L_\infty}$, $\|\eta^{-1}\|_{L_\infty}$,
$\|\nu \|_{L_\infty}$, $\|\nu^{-1}\|_{L_\infty}$, and the parameters of the lattice $\Gamma$. 
By \eqref{def} and \eqref{v_eq14},
\begin{align}
\label{flux2}
g^\eps b(\D) ( {\L}_\eps+I)^{-1} &= -i \begin{pmatrix} 
(\eta^\eps)^{-1} \rot \mu_0^{-1/2} (\L_\eps +I )^{-1}
\cr 
 \nu^\eps \div \mu_0^{1/2} (\L_\eps +I )^{-1}
\end{pmatrix},
\\
\label{flux3}
\wt{g}^\eps S_\eps b(\D) ( {\L}^0 +I)^{-1} &= -i \begin{pmatrix} 
((\eta^0)^{-1} + \Sigma^\eps) S_\eps \rot \mu_0^{-1/2} (\L^0 +I )^{-1}
\cr 
 \underline{\nu} S_\eps \div \mu_0^{1/2} (\L^0 +I )^{-1}
\end{pmatrix}.
\end{align}

Let us show that in estimate \eqref{flux1}  
the operator $S_\eps$ can be replaced by the identity; only the constant in estimate will change. 

\begin{lemma}\label{lem1.7}
For $\eps>0$ we have
\begin{equation}
\label{flux4}
\big\| \wt{g}^\eps (S_\eps -I) b(\D)( {\L}^0 +I)^{-1} \big\|_{L_2(\R^3) \to L_2(\R^3)} \le C' \eps. 
\end{equation}
The constant $C'$ depends only on $|\mu_0|$, $|\mu_0^{-1}|$, $\|\eta\|_{L_\infty}$, $\|\eta^{-1}\|_{L_\infty}$,
$\|\nu \|_{L_\infty}$, $\|\nu^{-1}\|_{L_\infty}$, and the parameters of the lattice $\Gamma$.
\end{lemma}

\begin{proof}
By \eqref{flux3}, the left-hand side of \eqref{flux4} does not exceed
\begin{equation}
\label{flux5}
\big\| g^0 (S_\eps -I) b(\D)  (\L^0 +I )^{-1} \big\|_{L_2(\R^3) \to L_2(\R^3)}
+ \big\| \Sigma^\eps (S_\eps -I) \rot \mu_0^{-1/2} (\L^0 +I )^{-1} \big\|_{L_2(\R^3) \to L_2(\R^3)}. 
\end{equation}
Using Proposition \ref{prop_Seps - I}, 
we estimate the first term in \eqref{flux5}: 
\begin{equation}
\label{flux5a}
\big\| g^0 (S_\eps -I) b(\D)  (\L^0 +I )^{-1} \big\|_{L_2(\R^3) \to L_2(\R^3)}
\le \eps \|g\|_{L_\infty} r_1 \big\| \D b(\D)  (\L^0 +I )^{-1} \big\|_{L_2(\R^3) \to L_2(\R^3)}.
\end{equation}
Recalling that the columns of the matrix $\Sigma^\eps$ are  $(\nabla \Phi_j)^\eps$, $j=1,2,3,$ we estimate 
the second term in \eqref{flux5}.
Applying Proposition \ref{prop_PSu}, and next Proposition  \ref{prop_Seps - I} and inequality  \eqref{S_eps <= 1},
we have:
\begin{equation}
\label{flux5b}
\begin{aligned}
 &\big\| (\nabla \Phi_j)^\eps (S_\eps -I) \rot \mu_0^{-1/2} (\L^0 +I )^{-1} \big\|_{L_2(\R^3) \to L_2(\R^3)} 
\\
&\le
\sqrt{\beta_1} \big\|  (S_\eps -I) \rot \mu_0^{-1/2} (\L^0 +I )^{-1} \big\|_{L_2(\R^3) \to L_2(\R^3)} 
\\
&\qquad+
\sqrt{\beta_2} \eps \|\Phi_j\|_{L_\infty} \big\|  (S_\eps -I) \D \rot \mu_0^{-1/2} (\L^0 +I )^{-1}\big \|_{L_2(\R^3) \to L_2(\R^3)}
\\
& \le \eps (\sqrt{\beta_1} r_1 + 2 \sqrt{\beta_2} \|\Phi_j\|_{L_\infty}) \big\| \D b(\D)  (\L^0 +I )^{-1} \big\|_{L_2(\R^3) \to L_2(\R^3)}.
\end{aligned}
\end{equation}
From \eqref{bb_eps} and \eqref{symbol_eff_est} it follows that
\begin{equation}
\label{flux5c}
\| \D b(\D)  (\L^0 +I )^{-1} \|_{L_2(\R^3) \to L_2(\R^3)} \le \sup_{\bxi \in \R^3} |\bxi b(\bxi) (a(\bxi) + \1)^{-1}| \le \alpha_1^{1/2} c_1^{-1}.
\end{equation}
As a result, inequalities  \eqref{flux5a}--\eqref{flux5c} together with Remark~\ref{LaUr} yield the required estimate \eqref{flux4}.
\end{proof}

Now, relations \eqref{flux1}--\eqref{flux4} imply the following result.

\begin{theorem}\label{th3} 
For $\eps>0$ we have  
$$
\big\| g^\eps b(\D) ( {\L}_\eps+I)^{-1} - \wt{g}^\eps b(\D) ( {\L}^0 +I)^{-1} \big\|_{L_2(\R^3)\to L_2(\R^3)}
\le C_3 \eps.
$$
In other words,
\begin{equation*}
\begin{split}
\big\| (\eta^\eps)^{-1} \rot \mu_0^{-1/2} (\L_\eps +I )^{-1}
- \big((\eta^0)^{-1} + \Sigma^\eps\big) \rot \mu_0^{-1/2} (\L^0 +I )^{-1}
 \big\|_{L_2(\R^3) \to L_2(\R^3)} 
\le  C_3 \eps,
\\
\big\| \nu^\eps \div \mu_0^{1/2} (\L_\eps +I )^{-1} - \underline{\nu} \div \mu_0^{1/2} (\L^0 +I )^{-1} 
\big\|_{L_2(\R^3) \to L_2(\R^3)} \le C_3 \eps. 
\end{split}
\end{equation*} 
The constant $C_3$ depends only on $|\mu_0|$, $|\mu_0^{-1}|$, $\|\eta\|_{L_\infty}$, $\|\eta^{-1}\|_{L_\infty}$,
$\|\nu \|_{L_\infty}$, $\|\nu^{-1}\|_{L_\infty}$, and the parameters of the lattice $\Gamma$. 
\end{theorem}

Now we distinguish particular cases. Taking Proposition \ref{prop_cases} and Remark \ref{rem_eta0} into account, we deduce the following result from Theorems \ref{th2} and \ref{th3}. 

\begin{proposition}

\noindent\emph{1)} Let $\eta^0= \overline{\eta}$, i.~e., the columns of the matrix $\eta(\x)$ are divergence free.
Then for $\eps >0$ we have  
$$
\big\| (\eta^\eps)^{-1} \rot \mu_0^{-1/2} (\L_\eps +I )^{-1}
- (\eta^0)^{-1} \rot \mu_0^{-1/2} (\L^0 +I )^{-1}
 \big\|_{L_2(\R^3) \to L_2(\R^3)} \le C_3 \eps.
$$

\noindent\emph{2)} Let $\eta^0= \underline{\eta}$, i.~e., the columns of the matrix $\eta(\x)^{-1}$ are potential.
Suppose, in addition, that $\nu(\x) = \operatorname{Const}$. Then the corrector \eqref{corr1} is equal to zero, and for $\eps >0$ we have 
$$
\big\| (\L_\eps +I)^{-1} -(\L^0 +I)^{-1} \big\|_{L_2(\R^3) \to H^1(\R^3)} \le C_2 \eps.
$$
\end{proposition}

\section{The second order model operator in a bounded domain\label{Sec2}} 

\subsection{Definition of the operator ${L}_\eps$\label{Sec2.2}}
Let $\mathcal{O}\subset\mathbb{R}^3$ be a bounded domain of class $C^{1,1}$. 
Let $\n(\x)$ be the unit outer normal vector to $\partial \O$ at the point $\x \in \partial \O$.
The projection of the vector-valued function $\u(\x)$ 
onto the normal vector on the boundary is denoted by
$$
\u_n(\x):= \langle \u(\x), \n(\x)\rangle,
$$ 
its tangential component is denoted by
$$
\u_\tau(\x):= \u(\x) - \u_n(\x) \n(\x).
$$

Suppose that the coefficients $\mu_0$, $\eta$, and $\nu$ satisfy the assumptions of Subsection~\ref{Sec1.2}.
In $L_2(\mathcal{O};\mathbb{C}^3)$, we consider the quadratic form 
\begin{equation}\label{form_l_eps}
\begin{split}
{l}_\eps[\f,\f]& :=
\intop_{\O} \big\langle g^\eps(\x) b(\D) \f, b(\D) \f \big\rangle \, d\x 
\\
&= \intop_{\O} \Big(\big \langle (\eta^\eps(\x))^{-1} \rot (\mu_0^{-1/2}\f), \rot (\mu_0^{-1/2}\f)\big\rangle
+ \nu^\eps(\x) \big|\div (\mu_0^{1/2}\f)\big|^2\Big) \, d\x, 
\end{split}
\end{equation}
defined on the domain 
\begin{equation}\label{Dom_form_l_eps}
\begin{aligned}
\Dom {l}_\eps = \big\{  &\f \in L_2(\O;\C^3):\ \rot (\mu_0^{-1/2}\f) \in L_2(\O;\C^3),
\\
 &\div (\mu_0^{1/2}\f) \in L_2(\O),\ (\mu_0^{1/2} \f)_n \vert_{\partial \O}=0 \big\}.
\end{aligned}
\end{equation}
Apriori, conditions from  \eqref{Dom_form_l_eps} on a vector-valued function $\f \in L_2(\O;\C^3)$ (in particular, the boundary condition) are understood in the generalized sense; see \cite{BS1, BS2} and Definition \ref{def1} below. 
Since ${\partial \O \in C^{1,1}}$, the set \eqref{Dom_form_l_eps} coincides with 
\begin{equation*}
\Dom {l}_\eps = \left\{ \f \in H^1(\O;\C^3):\  (\mu_0^{1/2} \f)_n \vert_{\partial \O}=0 \right\}.
\end{equation*}
Then the boundary condition can be understood in the sense of the trace theorem. Under our assumptions, the form \eqref{form_l_eps} is coercive. The following two-sided estimates hold:
\begin{equation}\label{form_l_eps_est_O}
{\mathfrak c}_1 \| \f \|^2_{H^1(\O)} \le {l}_\eps[\f,\f] + \|\f\|_{L_2(\O)}^2 \le {\mathfrak c}_2 \| \f \|^2_{H^1(\O)}, \quad \f \in \Dom l_\eps.
\end{equation}
The constant ${\mathfrak c}_1$ depends on $|\mu_0|$, $|\mu_0^{-1}|$, $\|\eta\|_{L_\infty}$, $\|\nu^{-1}\|_{L_\infty}$, and the domain $\O$, the constant 
${\mathfrak c}_2$ depends on $|\mu_0|$, $|\mu_0^{-1}|$, $\|\eta^{-1}\|_{L_\infty}$, $\|\nu\|_{L_\infty}$, and the domain $\O$.
These properties were checked in \cite[Theorem~2.3]{BS1} under the assumption that $\partial \O \in C^2$
and in \cite[Theorem~2.6]{F} under the assumption that $\partial \O \in C^{3/2+\delta}$, $\delta >0$.

Thus, the form \eqref{form_l_eps} is closed and nonnegative. A selfadjoint operator in $L_2(\O;\C^3)$ generated by this form is denoted by $L_\eps$. Formally, $L_\eps$ is given by the differential expression 
\begin{equation*}
{L}_\eps = \mu_0^{-1/2} \rot (\eta^\eps(\x))^{-1} \rot \mu_0^{-1/2} - \mu_0^{1/2} \nabla \nu^\eps(\x) \div \mu_0^{1/2}
\end{equation*}
with the boundary conditions
$$
(\mu_0^{1/2} \f)_n \vert_{\partial \O}=0, \quad \bigl((\eta^\eps)^{-1} \rot (\mu_0^{-1/2} \f)\bigr)_\tau \vert_{\partial \O}=0.
$$
The second condition is ``natural'' and is not reflected in the domain of the quadratic form $l_\eps$.

\begin{remark}
 In \cite{Su_SIAM},  when studying the general operators of the form
 $
b(\D)^* g^\eps(\x) b(\D)
$
 with the Neumann boundary condition, it was assumed that the rank of the symbol 
$b(\bxi)$ is maximal for $0 \ne \bxi \in \C^n$. This condition ensured the coercivity of the corresponding quadratic form on the class $H^1(\O;\C^n)$. 
In our case, this condition is not satisfied, though for  $0\ne \bxi \in \R^3$ the rank of the matrix  $b(\bxi)$ is maximal; see \eqref{rank_cond}. 
We emphasize that the form $l_\eps$ is coercive due to the boundary condition $(\mu_0^{1/2} \f)_n \vert_{\partial \O}=0$.
\end{remark}

\textit{Our goal} is to approximate the generalized solution of the problem   
\begin{equation}\label{phi_eps_pr}
\begin{aligned}
&\mu_0^{-1/2} \rot \big(\eta^\eps(\x)\big)^{-1} \rot \big(\mu_0^{-1/2} \f_\eps(\x)\big)
- \mu_0^{1/2} \nabla \nu^\eps(\x) \div \big(\mu_0^{1/2}\f_\eps(\x)\big) + \f_\eps(\x)= \FF(\x),\  \x \in \O;
\\
&(\mu_0^{1/2} \f_\eps)_n \vert_{\partial \O}=0, \quad \bigl((\eta^\eps)^{-1} \rot (\mu_0^{-1/2} \f_\eps)\bigr)_\tau \vert_{\partial \O}=0,
\end{aligned}
\end{equation}
for small $\eps$. Here $\FF \!\in\! L_2(\O;\C^3)$.
The solution is understood in the weak sense: $\f_\eps \!\in\! H^1(\O;\C^3)$, $(\mu_0^{1/2} \f_\eps)_n \vert_{\partial \O}=0$, and 
\begin{equation}\label{identity_est}
l_\eps [\f_\eps, \bzeta]\! + \!(\f_\eps,\bzeta)_{L_2(\O)}\! = \! (\FF,\bzeta)_{L_2(\O)},\quad 
\bzeta \!\in\! H^1(\O;\C^3), \quad (\mu_0^{1/2} \bzeta)_n \vert_{\partial \O}\!=\!0.
\end{equation}
Then $\f_\eps = (L_\eps +I)^{-1}\FF$. Thus, we are interested in the behavior of the resolvent 
$(L_\eps +I)^{-1}$ for small $\eps$.

\subsection{The effective operator ${L}^0$\label{Sec2.3}}
Suppose that the matrix $\eta^0$ is defined by \eqref{v_eq6}, \eqref{eta0}. 
Recall that $\underline{\nu}$ is the harmonic average of the coefficient $\nu(\x)$. 
Let $g^0$ be the matrix \eqref{v_eq15}. The effective operator $L^0$ is a selfadjoint operator in  $L_2(\O;\C^3)$
generared by the quadratic form 
\begin{equation}\label{form_l0}
\begin{aligned}
&{l}^0[\f,\f] := 
\intop_{\O} \langle g^0 b(\D) \f, b(\D) \f \rangle \, d\x
\\
&= \intop_{\O} \left( \langle (\eta^0)^{-1} \rot (\mu_0^{-1/2}\f), \rot (\mu_0^{-1/2}\f) \rangle + \underline{\nu} |\div (\mu_0^{1/2}\f)|^2\right) \, d\x,
\\
&\f \in H^1(\O;\C^3), \quad (\mu_0^{1/2} \f)_n \vert_{\partial \O}=0.  
\end{aligned}
\end{equation}
By \eqref{eff_coeff_est}, the form \eqref{form_l0} satisfies the estimates 
\begin{equation}\label{form_l0_est}
\begin{split}
{\mathfrak c}_1 \| \f \|^2_{H^1(\O)} \le {l}^0[\f,\f] + \|\f\|^2_{L_2(\O)} \le {\mathfrak c}_2 \| \f \|^2_{H^1(\O)}, 
\\
 \f \in H^1(\O;\C^3),\quad (\mu_0^{1/2} \f)_n \vert_{\partial \O}=0,
\end{split}
\end{equation}
with the same constants as in \eqref{form_l_eps_est_O}.

Due to the smoothness of the boundary, the following regularity property  holds:
 the operator $L^0$
is given by the differential expression  
\begin{equation*}
{L}^0 = \mu_0^{-1/2} \rot (\eta^0)^{-1} \rot \mu_0^{-1/2} - \mu_0^{1/2} \nabla \underline{\nu} \div \mu_0^{1/2}
\end{equation*}
on the domain
$$
\Dom L^0\! =\! \big\{ \f \in H^2(\O;\C^3), \ (\mu_0^{1/2} \f)_n \vert_{\partial \O}=0,\ 
\bigl((\eta^0)^{-1} \rot (\mu_0^{-1/2} \f)\bigr)_\tau \vert_{\partial \O}\!=0\big\}.
$$
Herewith,
\begin{equation}\label{L0_H2}
\| (L^0 + I)^{-1}\|_{L_2(\O) \to H^2(\O)} \le \widehat{c},  
\end{equation}
where the constant $\wh{c}$ depends on $|\mu_0|$, $|\mu_0^{-1}|$, $\| \eta \|_{L_\infty}$, 
$\| \eta^{-1} \|_{L_\infty}$, $\| \nu \|_{L_\infty}$, $\| \nu^{-1} \|_{L_\infty}$, and the domain~$\O$.

\begin{remark}
Under the assumption that $\partial \O \in C^{1,1}$ (and for sufficiently smooth coefficients), such regularity property 
for the solutions of the Dirichlet or Neumann problems for  
the second order strongly elliptic equations can be found, e.~g., in the book \cite[Chapter 4]{McL}.
The proof is based on the method of difference  quotients and essentialy relies on the coercivity condition for the quadratic form. In our case, the coefficients of the operator $L^0$ are constant and the coercivity condition~\eqref{form_l0_est} holds, but the boundary conditions are of mixed type. It is easy to check the regularity 
for the operator $L^0$ by the same method as before.  
\end{remark}

Let $\f_0$ be the solution of the ``homogenized''  problem  
\begin{equation}\label{phi_0_pr}
\begin{aligned}
&\mu_0^{-1/2} \rot \big(\eta^0\big)^{-1} \rot \big(\mu_0^{-1/2} \f_0(\x)\big)
- \mu_0^{1/2} \nabla \underline{\nu} \div \big(\mu_0^{1/2}\f_0(\x)\big) + \f_0(\x)= \FF(\x),\quad  \x \in \O;
\\
&(\mu_0^{1/2} \f_0)_n \vert_{\partial \O}=0, \quad \bigl((\eta^0)^{-1} \rot (\mu_0^{-1/2} \f_0)\bigr)_\tau \vert_{\partial \O}=0.
\end{aligned}
\end{equation}
In other words, the function $\f_0 \in H^1(\O;\C^3)$ satisfies the boundary condition $(\mu_0^{1/2} \f_0)_n \vert_{\partial \O}=0$ and the identity 

\begin{equation}\label{phi_0_identity}
l^0 [\f_0, \bzeta] + (\f_0, \bzeta)_{L_2(\O)} = (\FF, \bzeta)_{L_2(\O)}, 
\quad \bzeta \in H^1(\O;\C^3), \quad (\mu_0^{1/2} \bzeta)_n \vert_{\partial \O}=0.
\end{equation}
Then $\f_0 = (L^0 +I)^{-1}\FF$. Estimate \eqref{L0_H2} means that $\f_0 \in H^2(\O;\C^3)$ and 
\begin{equation}\label{phi_0_est}
\| \f_0 \|_{H^2(\O)} \le \wh{c} \| \FF \|_{L_2(\O)}.
\end{equation}

\subsection{Estimates in the neighborhood of the boundary}

We put
$$
(\partial\mathcal{O})_{\varepsilon} :=\left\lbrace \mathbf{x}\in \mathbb{R}^d : \mathrm{dist}\,\lbrace \mathbf{x};\partial\mathcal{O}\rbrace <\varepsilon \right\rbrace, \quad \eps >0. 
$$
We choose the numbers  $\varepsilon _0, \varepsilon _1\in (0,1]$ satisfying the following condition.

\begin{condition}
\label{condition varepsilon}
The number $\varepsilon _0\in (0,1]$ is such that the strip  
$(\partial\mathcal{O})_{\varepsilon_0}$ can be covered by a finite number of open sets admitting diffeomorphisms of class  $C^{0,1}$ rectifying the boundary $\partial\mathcal{O}$.
Let $\varepsilon _1 :=\varepsilon _0 (1+r_1)^{-1}$, where $2r_1=\mathrm{diam}\,\Omega$.
\end{condition}

Clearly, $\varepsilon _1$ depends only on the domain $\mathcal{O}$ and the parameters of the lattice $\Gamma$.
Note that Condition~\ref{condition varepsilon} is ensured only by the Lipschitz property of the boundary. We have imposed a more restrictive assumption $\partial\mathcal{O}\in C^{1,1}$ in order to ensure estimate~\eqref{L0_H2}.

The following statements were checked in~\cite[Section~5]{PSu}; Lemma~\ref{Lemma 3.6 from Su15} is similar to  Lemma~2.6 from \cite{ZhPas1}.

\begin{lemma}
\label{lemma ots int O_eps B_eps}
Suppose that Condition~\textnormal{\ref{condition varepsilon}} is satisfied.
Let $0< \varepsilon\leqslant\varepsilon _0$. Denote $B_\eps:= \O \cap (\partial \O)_\eps$. 

\noindent \emph{1)} For any function $u\in H^1(\O)$ we have
\begin{equation*}
\int\limits_{B_\varepsilon }\vert u\vert ^2 \,d\mathbf{x}\leqslant \beta\varepsilon \Vert u\Vert _{H^1(\O)}\Vert u\Vert _{L_2(\O)}.
\end{equation*}

\noindent \emph{2)}  
For any function $u\in H^1(\R^3)$ we have 
 
\begin{equation*}
\int\limits_{(\partial \O)_\varepsilon }\vert u\vert ^2 \,d\mathbf{x}\leqslant \beta\varepsilon \Vert u\Vert _{H^1(\R^3)}\Vert u\Vert _{L_2(\R^3)}.
\end{equation*}

\noindent
 The constant $\beta$ depends only on the domain $\mathcal{O}$.
\end{lemma}

\begin{lemma}
\label{Lemma 3.6 from Su15}
Suppose that Condition~\textnormal{\ref{condition varepsilon}} is satisfied.
Let  $h(\mathbf{x})$ be a \hbox{$\Gamma$-periodic} function in~$\mathbb{R}^3$ such that $h\in L_2(\Omega)$.
Let $S_\varepsilon$ be the operator~\eqref{S_eps}. Denote~$\beta _* :=\beta (1+r_1)$, where $2r_1=\mathrm{diam}\,\Omega$. Then for $0<\varepsilon\leqslant\varepsilon _1$ and $\mathbf{u}\in H^1(\mathbb{R}^3;\mathbb{C}^k)$ we have 
\begin{equation*}
\int\limits_{(\partial\mathcal{O})_\varepsilon}\vert h^\varepsilon (\mathbf{x})\vert ^2 \vert (S_\varepsilon \mathbf{u})(\mathbf{x})\vert ^2\,d\mathbf{x}
\leqslant \beta _*\varepsilon\vert \Omega\vert ^{-1}\Vert h\Vert ^2_{L_2(\Omega)}\Vert \mathbf{u}\Vert _{H^1(\mathbb{R}^3)}\Vert \mathbf{u}\Vert _{L_2(\mathbb{R}^3)}.
\end{equation*}
\end{lemma}

\section{The results for the model second order equation\\ in a bounded domain\label{Sec3}}

\subsection{Approximation of the resolvent of the operator $L_\eps$}

Now, we formulate our main results about approximation of the solution of problem \eqref{phi_eps_pr}. 
For convenience of further references,  the following set of the parameters is called the ``problem data'':
\begin{equation}\label{data}
\begin{split}
&|\mu_0|,\  |\mu_0^{-1}|,\  \| \eta \|_{L_\infty}, \ \| \eta^{-1} \|_{L_\infty},\  \| \nu \|_{L_\infty},\  \| \nu^{-1} \|_{L_\infty};
\\
&\text{the parameters of the lattice } \Gamma; \ \text{and the domain} \ \O.
\end{split}
\end{equation}

\begin{theorem}\label{th1_O}
Let  $\f_\eps$ be the solution of problem \eqref{phi_eps_pr}, and let $\f_0$ be the solution of the homogenized problem \eqref{phi_0_pr} with $\FF \in L_2(\O;\C^3)$. 
Suppose that the number $\eps_1$ satisfies Condition~\textnormal{\ref{condition varepsilon}}. Then for $0 < \eps \le \eps_1$ we have
\begin{equation}\label{res1}
\|  \f_\eps - \f_0 \|_{L_2(\O)} \le {\mathcal C}_1 \eps \| \FF \|_{L_2(\O)}. 
\end{equation}
In operator terms,
\begin{equation*}
\big\|  (L_\eps +I)^{-1} - (L^0 +I)^{-1}  \big\|_{L_2(\O) \to L_2(\O)} \le {\mathcal C}_1 \eps. 
\end{equation*}
The constant $\mathcal{C}_1$ depends only on the problem data \eqref{data}.
\end{theorem}

To approximate the solution in $H^1(\O;\C^3)$, we need to introduce a corrector. 
We fix a linear continuous extension operator  
\begin{equation*}
P_\O: H^s(\O;\C^3) \to H^s(\R^3;\C^3),\quad s=0,1,2. 
\end{equation*}
Such an operator exists for any bounded domain with Lipschitz boundary
(see, e.~g., \cite{St}). Denote  
\begin{equation}\label{3.5}
\| P_\O\|_{ H^s(\O) \to H^s(\R^3)} =: C_\O^{(s)} ,\quad s=0,1,2. 
\end{equation}
The constants $C_\O^{(s)}$ depend only on the domain $\O$.
Next, let $[\Lambda^\eps]$ be the operator of multiplication by the matrix-valued function $\Lambda ({\eps}^{-1}\x)$,
and let $R_\O$ be the restriction operator of functions in $\R^3$ onto the domain $\O$. Let $S_\eps$ be the Steklov smoothing operator; see \eqref{S_eps}.  We introduce a corrector
$$
K_\eps := R_\O [\Lambda^\eps] S_\eps b(\D) P_\O (L^0 + I)^{-1}.
$$
The operator
$
b(\D) P_\O (L^0 + I)^{-1}
$
is continuous from $L_2(\O;\C^3)$ to $H^1(\R^3;\C^4)$.
As has been already mentioned, the operator $[\Lambda^\eps] S_\eps$ is continuous from  $H^1(\R^3;\C^4)$ to $H^1(\R^3;\C^3)$. Hence, the corrector $K_\eps$ is a continuous mapping of $L_2(\O;\C^3)$ to $H^1(\O;\C^3)$.
Using \eqref{def} and \eqref{v_eq13a}, we write the corrector as 
\begin{equation}\label{corr1_OO}
K_\eps = R_\O \left( \mu_0^{-1/2}\Psi^\eps S_\eps \rot \mu_0^{-1/2} + \mu_0^{1/2} (\nabla \rho)^\eps S_\eps \div \mu_0^{1/2}
\right) P_\O (L^0 +I)^{-1}. 
\end{equation}

Let $\f_0$ be the solution of problem \eqref{phi_0_pr}. We put $\wt{\f}_0 := P_\O \f_0$ and
\begin{equation}\label{corr1_O}
\begin{split}
\wt{\bpsi}_\eps(\x) &:= \wt{\f}_0 (\x)+ \eps \mu_0^{-1/2}\Psi^\eps(\x) (S_\eps \rot \mu_0^{-1/2} \wt{\f}_0)(\x) 
\\
&\qquad\qquad\,
+ \eps \mu_0^{1/2} (\nabla \rho)^\eps(\x) (S_\eps \div \mu_0^{1/2} \wt{\f}_0)(\x), \quad \x \in \R^3,
\\
\bpsi_\eps &:= \wt{\bpsi}_\eps \vert_{\O}.
\end{split}
\end{equation} 
Then
\begin{equation}\label{corr2_O}
{\bpsi}_\eps = {\f}_0 + \eps \Lambda^\eps S_\eps b(\D) \wt{\f}_0 = (L^0 +I)^{-1} \FF + \eps K_\eps \FF.  
\end{equation}

\begin{theorem}\label{th2_O}
Suppose that the assumptions of Theorem \emph{\ref{th1_O}} are satisfied. Let~$\bpsi_\eps$ be defined by  \eqref{corr1_O}. Then for $0 < \eps \le \eps_1$ we have
\begin{equation}\label{res3}
\|  \f_\eps - \bpsi_\eps \|_{H^1(\O)} \le {\mathcal C}_2 \eps^{1/2} \| \FF \|_{L_2(\O)}. 
\end{equation}
In operator terms,
\begin{equation*}
\|  (L_\eps +I)^{-1} - (L^0 +I)^{-1}  -\eps K_\eps \|_{L_2(\O) \to H^1(\O)} \le {\mathcal C}_2 \eps^{1/2}. 
\end{equation*}
 The constant $\mathcal{C}_2$ depends only on the problem data \eqref{data}.
\end{theorem}

\begin{theorem}\label{th3_O}
Suppose that the assumptions of Theorem~\emph{\ref{th1_O}} are satisfied.
We put
$$
\u_\eps := (\eta^\eps)^{-1} \rot \mu_0^{-1/2}\f_\eps,\quad
\u_0 := (\eta^0)^{-1} \rot \mu_0^{-1/2}\f_0.
$$
 Then for $0 < \eps \le \eps_1$ we have
 \begin{equation}\label{res5}
\begin{split}
\|  \u_\eps - \u_0 - \Sigma^\eps  \rot ( \mu_0^{-1/2} {\f}_0 ) \|_{L_2(\O)} 
&\le \mathcal{C}_3 \eps^{1/2}\| \FF \|_{L_2(\O)}, 
\\
\| \nu^\eps \div (\mu_0^{1/2} \f_\eps) - \underline{\nu} \div (\mu_0^{1/2} \f_0) 
\|_{L_2(\O)} &\le {\mathcal C}_3 \eps^{1/2} \|\FF\|_{L_2(\O)}.
\end{split}
\end{equation}
 The constant $\mathcal{C}_3$ depends only on the problem data \eqref{data}.
\end{theorem}

Now, we distinguish the special cases. By Proposition \ref{prop_cases} and Remark \ref{rem_eta0}, Theorems \ref{th2_O} and \ref{th3_O}
directly imply the following statement.

\begin{proposition}

\noindent\emph{1)} Suppose that $\eta^0= \overline{\eta}$, i.~e., the columns of the matrix $\eta(\x)$ are divergence free.
Then for  $0< \eps \le \eps_1$ we have
$$
\| \u_\eps - \u_0  \|_{L_2(\O)} \le {\mathcal C}_3 \eps^{1/2} \|\FF\|_{L_2(\O)}.
$$

\noindent\emph{2)} Suppose that $\eta^0= \underline{\eta}$, i.~e., the columns of the matrix  $\eta(\x)^{-1}$ are potential.
Suppose, in addition, that $\nu(\x) = \operatorname{Const}$. Then the corrector \eqref{corr1_OO} is equal to zero and for 
 $0< \eps \le \eps_1$ we have
$$
\| \f_\eps - \f_0 \|_{H^1(\O)} \le \mathcal{C}_2 \eps^{1/2} \|\FF \|_{L_2(\O)}.
$$
\end{proposition}

\subsection{The first step of the proof. The associated problem in  $\R^3$}
Obviously, we have $\|(L^0+I)^{-1}\|_{L_2(\O) \to L_2(\O)} \le 1$, whence $\|\f_0\|_{L_2(\O)} \le \|\FF\|_{L_2(\O)}$. By  \eqref{phi_0_est} and
 \eqref{3.5}, 
\begin{align}\label{phi_tilde_1}
\| \wt{\f}_0 \|_{L_2(\R^3)}  &\le C_\O^{(0)} \|\FF\|_{L_2(\O)},
\\
\label{phi_tilde_2}
\| \wt{\f}_0 \|_{H^2(\R^3)}  &\le C_\O^{(2)} \wh{c} \|\FF\|_{L_2(\O)}.
\end{align}
We put
\begin{equation}\label{ass_problem_rhs}
\wt{\FF}:= {\L}^0 \wt{\f}_0 + \wt{\f}_0.
\end{equation}
Then $\wt{\FF}\in L_2(\R^3;\C^3)$ and $\wt{\FF}\vert_{ \O} = \FF$. By \eqref{symbol_eff_est}, \eqref{phi_tilde_1}, and \eqref{phi_tilde_2},
\begin{equation}\label{FF_est}
\| \wt{\FF}\|_{L_2(\R^3)} \le c_2 \| \wt{\f}_0\|_{H^2(\R^3)} +  \| \wt{\f}_0\|_{L_2(\R^3)} \le \mathcal{C}_4 \|\FF\|_{L_2(\O)},
\end{equation}
where $\mathcal{C}_4 = c_2 \wh{c} C_\O^{(2)} + C_\O^{(0)}$.
We also need the following inequality which directly follows from  \eqref{ass_problem_rhs} and \eqref{FF_est}:
\begin{equation}\label{3.15a}
\l^0[\wt{\f}_0, \wt{\f}_0]\le \| \wt{\FF}\|^2_{L_2(\R^3)} \le {\mathcal C}^2_4  \|\FF\|^2_{L_2(\O)}.
\end{equation}

Let $\wt{\f}_\eps\in H^1(\R^3;\C^3)$ be the generalized solution of the following equation in $\R^3$:
\begin{equation*}
 \L_\eps \wt{\f}_\eps   + \wt{\f}_\eps = \wt{\FF},
\end{equation*}
i.~e., $\wt{\f}_\eps = (\L_\eps +I)^{-1} \wt{\FF}$. We apply Theorems \ref{th1}, \ref{th2}, and \ref{th3}.
Using also \eqref{FF_est}, we arrive at the estimates 
\begin{align}
\label{ass1}
\| \wt{\f}_\eps - \wt{\f}_0 \|_{L_2(\R^3)} \le C_1 \eps \| \wt{\FF} \|_{L_2(\R^3)} \le C_1 {\mathcal C}_4\eps \| {\FF} \|_{L_2(\O)},
\\
\label{ass2}
\| \wt{\f}_\eps - \wt{\bpsi}_\eps \|_{H^1(\R^3)} \le C_2 \eps \| \wt{\FF} \|_{L_2(\R^3)} 
\le C_2 {\mathcal C}_4 \eps \| {\FF} \|_{L_2(\O)},
\\
\label{ass3}
\| g^\eps b(\D)\wt{\f}_\eps -  \wt{g}^\eps b(\D) \wt{\f}_0 \|_{L_2(\R^3)} \le C_3 \eps \| \wt{\FF} \|_{L_2(\R^3)} 
\le C_3 {\mathcal C}_4 \eps \| {\FF} \|_{L_2(\O)}.
\end{align}

\subsection{The second step of the proof. Introduction of the correction term $\s_\eps$}

Now, we introduce the  ``boundary layer correction term''  $\s_\eps \in H^1(\O;\C^n)$, as the function satisfying the following identity and boundary condition:
\begin{equation}\label{popravka}
\begin{split}
(g^\eps b(\D) \s_\eps,& b(\D) \bzeta)_{L_2(\O)} + (\s_\eps,  \bzeta)_{L_2(\O)} 
= (\wt{g}^\eps b(\D) \f_0, b(\D) \bzeta)_{L_2(\O)} - (\FF, \bzeta)_{L_2(\O)}+ (\f_0,  \bzeta)_{L_2(\O)},
\\
&\qquad\qquad \forall \bzeta \in H^1(\O;\C^3),  \quad (\mu_0^{1/2} \bzeta)_n \vert_{\partial \O} =0,
\\
 (\mu_0^{1/2} \s_\eps)_n &\vert_{\partial \O} = \eps (\mu_0^{1/2} \Lambda^\eps S_\eps b(\D) \wt{\f}_0)_n \vert_{\partial \O}.
\end{split}
\end{equation}

Let us show that taking $\s_\eps$  into account allows us to obtain approximation of the solution $\f_\eps$  in the $H^1$-norm with an error of sharp order $O(\eps)$.

\begin{theorem}
For $\eps>0$ we have
\begin{equation}\label{popravka2}
\|  \f_\eps - {\bpsi}_\eps + \s_\eps \|_{H^1(\O)} \le \mathcal{C}_5 \eps \| \FF \|_{L_2(\O)}.
\end{equation}
The constant $\mathcal{C}_5$ depends only on the problem data \eqref{data}.
\end{theorem}

\begin{proof}
Denote $\VV_\eps:= \f_\eps - {\bpsi}_\eps + \s_\eps$. Then from \eqref{identity_est}, \eqref{corr2_O}, \eqref{popravka}, and the boundary 
conditions $(\mu_0^{1/2} \f_\eps)_n\vert_{\partial \O}=0$, $(\mu_0^{1/2} \f_0)_n\vert_{\partial \O} =0$ it follows that 
 $\VV_\eps \in H^1(\O;\C^3)$, $(\mu_0^{1/2} \VV_\eps)_n \vert_{\partial \O} =0$, and
\begin{equation}\label{popravka3}
\begin{split}
l_\eps[\VV_\eps, \bzeta]& + (\VV_\eps, \bzeta)_{L_2(\O)}
=
(\wt{g}^\eps b(\D) \f_0 - g^\eps b(\D) {\bpsi}_\eps, b(\D) \bzeta)_{L_2(\O)} 
+ (\f_0 - {\bpsi}_\eps,  \bzeta)_{L_2(\O)}, 
\\
& \forall \bzeta \in H^1(\O;\C^3),  \quad (\mu_0^{1/2} \bzeta)_n \vert_{\partial \O} =0.
\end{split}
\end{equation}
The first term on the right can be written as
$$
(\wt{g}^\eps b(\D) \wt{\f}_0 - g^\eps b(\D) \wt{\f}_\eps, b(\D) \bzeta)_{L_2(\O)} + ({g}^\eps b(\D) (\wt{\f}_\eps - \wt{\bpsi}_\eps), b(\D) \bzeta)_{L_2(\O)}.
$$
By \eqref{form_l_eps_est}, \eqref{ass2}, and \eqref{ass3}, it does not exceed
$$
\begin{aligned}
&\|\wt{g}^\eps b(\D) \wt{\f}_0 - g^\eps b(\D) \wt{\f}_\eps\|_{L_2(\R^3)} \| b(\D) \bzeta\|_{L_2(\O)} 
\\
&
+
\left( \l_\eps[\wt{\f}_\eps - \wt{\bpsi}_\eps,\wt{\f}_\eps - \wt{\bpsi}_\eps]\right)^{1/2} \left( l_\eps[\bzeta, \bzeta]\right)^{1/2}
\le {\mathcal C}_5' \eps \|\FF\|_{L_2(\O)}\left( l_\eps[\bzeta, \bzeta]\right)^{1/2},
\end{aligned}
$$
where ${\mathcal C}_5' = {\mathcal C}_4 \left( \|g^{-1}\|_{L_\infty}^{1/2} C_3 + \sqrt{c_2} C_2\right)$. The second term in the right-hand side of 
\eqref{popravka3} can be written as 
$(\wt{\f}_0 - \wt{\f}_\eps,  \bzeta)_{L_2(\O)}+(\wt{\f}_\eps - \wt{\bpsi}_\eps,  \bzeta)_{L_2(\O)}$.
By \eqref{ass1} and \eqref{ass2}, 
it does not exceed  ${\mathcal C}_5'' \eps \|\FF\|_{L_2(\O)} \|\bzeta\|_{L_2(\O)}$,  where
\hbox{${\mathcal C}_5'' = {\mathcal C}_4 (C_1 + C_2)$}. As a result, we see that the right-hand side of identity \eqref{popravka3} 
is majorated by  
$\check{\mathcal C}_5 \eps \|\FF\|_{L_2(\O)} \left( l_\eps[\bzeta, \bzeta]+ \|\bzeta\|^2_{L_2(\O)}\right)^{1/2}$,
where $\check{\mathcal C}_5^2 = ({\mathcal C}_5')^2 + ({\mathcal C}_5'')^2$. 

Substituting $\bzeta = \VV_\eps$ in \eqref{popravka3} and using the obtained estimate, we arrive at the inequality  
$$
\left( l_\eps[\VV_\eps, \VV_\eps]+ \|\VV_\eps\|^2_{L_2(\O)}\right)^{1/2} \le \check{\mathcal C}_5 \eps \| \FF \|_{L_2(\O)}.
$$
Together with the lower estimate \eqref{form_l_eps_est_O}, this implies the required inequality  
\eqref{popravka2} with the constant ${\mathcal C}_5 = \check{\mathcal C}_5 {\mathfrak c}_1^{-1/2}$.
\end{proof}

\noindent\textbf{Conclusions.}
1) From \eqref{popravka2} it follows that  
\begin{equation}\label{vyvod1}
\| \f_\eps - \bpsi_\eps \|_{H^1(\O)} \le 
 \mathcal{C}_5 \eps \| \FF \|_{L_2(\O)} + \| \s_\eps \|_{H^1(\O)}.
\end{equation}
So, for the proof of  Theorem \ref{th2_O}, we need to prove a suitable estimate for the norm $\| \s_\eps \|_{H^1(\O)}$.

2) From \eqref{corr2_O} and \eqref{popravka2} it follows that 
\begin{equation}\label{vyvod2}
\| \f_\eps - \f_0 \|_{L_2(\O)} \le 
\mathcal{C}_5 \eps \| \FF \|_{L_2(\O)} + \eps \| \Lambda^\eps S_\eps b(\D) \wt{\f}_0 \|_{L_2(\R^3)}+\| \s_\eps \|_{L_2(\O)}.
\end{equation}
By Proposition \ref{prop f^eps S_eps} and estimates \eqref{Lambda_est}, \eqref{3.15a}, 
\begin{equation}\label{3.25a}
\| \Lambda^\eps S_\eps b(\D) \wt{\f}_0 \|_{L_2(\R^3)} \le {\mathfrak C}_\Lambda {\mathcal C}_4 \| g^{-1}\|_{L_\infty}^{1/2} \| \FF \|_{L_2(\O)}.
\end{equation}
Together with \eqref{vyvod2}, this yields 
\begin{equation}\label{vyvod3}
\| \f_\eps - \f_0 \|_{L_2(\O)} \le 
\mathcal{C}_6 \eps \| \FF \|_{L_2(\O)} +\| \s_\eps \|_{L_2(\O)},
\end{equation}
where $\mathcal{C}_6 = \mathcal{C}_5 + {\mathfrak C}_\Lambda {\mathcal C}_4 \| g^{-1}\|_{L_\infty}^{1/2}$.
Hence, to prove Theorem \ref{th1_O}, we have to estimate the norm $\| \s_\eps \|_{L_2(\O)}$ in appropriate way.

\section{Estimation of the correction term. Proof of Theorems \ref{th1_O}--\ref{th3_O} \label{Sec4}}

First, we estimate the  $H^1$-norm of the correction term $\s_\eps$ and prove Theorem \ref{th2_O} and also Theorem \ref{th3_O}. 
Next, using the already proved Theorem \ref{th2_O} and the duality arguments, 
we estimate the $L_2$-norm of the correction term $\s_\eps$ and prove Theorem~\ref{th1_O}.

\subsection{Estimate for the correction term in  $H^1(\O)$. Proof of Theorem \ref{th2_O}}

We rewrite identity \eqref{popravka}, using \eqref{phi_0_identity}:
\begin{equation}\label{popravka_identity}
\begin{split}
(g^\eps b(\D) \s_\eps, b(\D) \bzeta)_{L_2(\O)}& + (\s_\eps,  \bzeta)_{L_2(\O)} 
= ((\wt{g}^\eps - g^0) b(\D) \f_0, b(\D) \bzeta)_{L_2(\O)} =:\mathcal{I}_\eps[\bzeta], 
\\
& \forall \bzeta \in H^1(\O;\C^3),  \quad (\mu_0^{1/2} \bzeta)_n \vert_{\partial \O} =0.
\end{split}
\end{equation}
According to  \eqref{def}, \eqref{v_eq14}, and \eqref{v_eq15}, 
\begin{equation}\label{I_eps}
\mathcal{I}_\eps[\bzeta] = (\Sigma^\eps \rot \mu_0^{-1/2} \f_0, \rot \mu_0^{-1/2} \bzeta)_{L_2(\O)}.
\end{equation}

\begin{lemma}
For $0< \eps \le \eps_0$ we have 
\begin{equation}\label{I_eps_estimate}
|\mathcal{I}_\eps[\bzeta]| \le \mathcal{C}_7 \eps^{1/2} \| \FF \|_{L_2(\O)} (l_\eps[\bzeta, \bzeta])^{1/2},
\quad \bzeta \in H^1(\O;\C^3),  \quad (\mu_0^{1/2} \bzeta)_n \vert_{\partial \O} =0.
\end{equation}
The constant $\mathcal{C}_7$ depends only on the problem data \eqref{data}. 
\end{lemma}

\begin{proof}
Recall that $\Sigma(\x)$ is the matrix with the columns $\nabla \Phi_j(\x)$, $j=1,2,3$.
Hence, the matrix $\Sigma^\eps(\x)$ has the columns  $(\nabla \Phi_j)^\eps(\x) = \eps \nabla \Phi_j^\eps(\x)$, $j=1,2,3$.
The components of the vector-valued function $\rot \mu_0^{-1/2} \f_0$ are denoted by 
$[\rot \mu_0^{-1/2} \f_0]_j$, $j=1,2,3$. Then  
\begin{align*}
\Sigma^\eps \rot \mu_0^{-1/2} \f_0& = \eps \sum_{j=1}^3 (\nabla \Phi^\eps_j) [\rot \mu_0^{-1/2} \f_0]_j 
\\
&=
\eps \sum_{j=1}^3 \left( \nabla\left(\Phi^\eps_j [\rot \mu_0^{-1/2} \f_0]_j\right) - \Phi^\eps_j \nabla [\rot \mu_0^{-1/2} \f_0]_j \right).
\end{align*}
Together with \eqref{I_eps}, this implies  
\begin{align}\label{I_eps=sum}
\mathcal{I}_\eps[\bzeta] &= \mathcal{I}^{(1)}_\eps[\bzeta] + \mathcal{I}^{(2)}_\eps[\bzeta],
\\
\label{I1_eps}
\mathcal{I}^{(1)}_\eps[\bzeta] &:= \eps \sum_{j=1}^3 \left(\nabla\left(\Phi^\eps_j [\rot \mu_0^{-1/2} \f_0]_j\right), \rot \mu_0^{-1/2} \bzeta\right)_{L_2(\O)},
\\
\label{I2_eps}
\mathcal{I}^{(2)}_\eps[\bzeta] &:= - \eps \sum_{j=1}^3 (\Phi^\eps_j \nabla [\rot \mu_0^{-1/2} \f_0]_j, \rot \mu_0^{-1/2} \bzeta)_{L_2(\O)}.
\end{align}
By \eqref{phi_0_est} and the boundedness of $\Phi_j$ (see Remark \ref{LaUr}), the term \eqref{I2_eps} admits the estimate 
\begin{equation}\label{I2_eps_est}
|\mathcal{I}^{(2)}_\eps[\bzeta]| \le \mathcal{C}_7' \eps \| \FF \|_{L_2(\O)} (l_\eps[\bzeta, \bzeta])^{1/2},
\end{equation}
where the constant ${\mathcal C}_7'$ depends only on the problem data \eqref{data}. We have taken into account the obvious inequality
\begin{equation}\label{4.7a}
\| \rot \mu_0^{-1/2} \bzeta\|_{L_2(\O)} \le \|\eta\|^{1/2}_{L_\infty} (l_\eps[\bzeta, \bzeta])^{1/2}.
\end{equation}

Let $0< \eps \le \eps_0$. We fix a cut-off function $\theta_\eps(\x)$ in $\R^3$ such that
\begin{equation}\label{srezka}
\begin{split}
&\theta_\eps \in C_0^\infty(\R^3); \quad \operatorname{supp}\, \theta_\eps \subset (\partial \O)_\eps; \quad 
0\le \theta_\eps(\x) \le 1;
\\
&\theta_\eps(\x) =1 \ \text{for} \ \x \in \partial \O;\quad \eps |\nabla \theta_\eps| \le \kappa = \text{Const}.
\end{split}
\end{equation}
We put
\begin{equation}\label{4.8a}
{\mathbf f}_{j,\eps} := \eps \nabla\left( \theta_\eps \Phi^\eps_j [\rot \mu_0^{-1/2} \f_0]_j\right), \quad j=1,2,3,
\end{equation}
and represent the term \eqref{I1_eps} as 
\begin{equation}
\label{I1_eps=}
\mathcal{I}^{(1)}_\eps[\bzeta] = \sum_{j=1}^3 ({\mathbf f}_{j,\eps}, \rot \mu_0^{-1/2} \bzeta)_{L_2(\O)}. 
\end{equation}
Here we have used the identity
$$
\left(\nabla\left( (1- \theta_\eps) \Phi^\eps_j [\rot \mu_0^{-1/2} \f_0]_j\right), \rot \mu_0^{-1/2} \bzeta\right)_{L_2(\O)} =0,
$$
which can be checked by integration by parts and using the identity $\div \rot =0$ (when checking, we can assume that  $\bzeta \in H^2(\O;\C^3)$). 

It remains to estimate the term \eqref{I1_eps=}. By \eqref{srezka}, \eqref{4.8a}, and Remark~\ref{LaUr}, we have
\begin{equation}\label{I1_eps_2}
\begin{split}
& \| {\mathbf f}_{j,\eps} \|_{L_2(\O)}
\le \kappa \| \Phi_j \|_{L_\infty} \| [\rot \mu_0^{-1/2} \f_0]_j \|_{L_2(B_\eps)}
\\
&+ \| \theta_\eps (\nabla \Phi_j)^\eps [\rot \mu_0^{-1/2} \f_0]_j \|_{L_2(\O)}
+ \eps \| \Phi_j \|_{L_\infty}  \| \nabla [\rot \mu_0^{-1/2} \f_0]_j \|_{L_2(\O)}.
\end{split}
\end{equation}
The first summand in \eqref{I1_eps_2} does not exceed  $C \eps^{1/2} \|\FF\|_{L_2(\O)}$, due to Lemma~\ref{lemma ots int O_eps B_eps} 
and estimate~\eqref{phi_0_est}.
 By \eqref{phi_0_est}, the third summand in \eqref{I1_eps_2} is majorated by $C \eps \|\FF\|_{L_2\!(\O)}$. 
To estimate the second term in  \eqref{I1_eps_2}, we apply Proposition \ref{prop_PSu} and \eqref{srezka}:
\begin{align*}
&\| \theta_\eps (\nabla \Phi_j)^\eps [\rot \mu_0^{-1/2} \f_0]_j \|_{L_2(\O)}
\le \| \theta_\eps (\nabla \Phi_j)^\eps [\rot \mu_0^{-1/2} \wt{\f}_0]_j \|_{L_2(\R^3)}
\\
&\le \sqrt{\beta_1} \|  [\rot \mu_0^{-1/2} \wt{\f}_0]_j \|_{L_2((\partial \O)_\eps)}
\!+\! \sqrt{\beta_2} \eps \|\Phi_j\|_{L_\infty} \| \nabla \!\Big( \theta_\eps [\rot \mu_0^{-1/2} \wt{\f}_0]_j\Big) \|_{L_2(\R^3)}
\\
&\le  \Big( \sqrt{\beta_1}\! +\! \sqrt{\beta_2} \|\Phi_j\|_{L_\infty} \kappa \Big)\|  [\rot \mu_0^{-1/2} \wt{\f}_0]_j \|_{L_2((\partial \O)_\eps)}
+ \sqrt{\beta_2} \eps \|\Phi_j\|_{L_\infty} \| \nabla [\rot \mu_0^{-1/2} \wt{\f}_0]_j \|_{L_2(\R^3)}.
\end{align*}
By Lemma \ref{lemma ots int O_eps B_eps} and estimate \eqref{phi_tilde_2}, the first term on the right does not exceed $C \eps^{1/2} \|\FF\|_{L_2(\O)}$.
From \eqref{phi_tilde_2} it follows that the second term is estimated by $C \eps \|\FF\|_{L_2(\O)}$.
We arrive at 
\begin{equation}\label{fj_eps_est}
\sum_{j=1}^3 \| {\mathbf f}_{j,\eps} \|_{L_2(\O)}  \le \mathcal{C}_7'' \eps^{1/2} \| \FF \|_{L_2(\O)},
\end{equation}
where the constant ${\mathcal C}_7''$ depends only on the problem data \eqref{data}. 
Hence, by \eqref{4.7a}, the term \eqref{I1_eps=} satisfies 
\begin{equation}\label{I1_eps_est}
|\mathcal{I}^{(1)}_\eps[\bzeta]| \le \mathcal{C}_7'' \|\eta\|_{L_\infty}^{1/2} \eps^{1/2} \| \FF \|_{L_2(\O)} (l_\eps[\bzeta, \bzeta])^{1/2}.
\end{equation}
Now, relations \eqref{I_eps=sum}, \eqref{I2_eps_est}, and \eqref{I1_eps_est} imply the required inequality \eqref{I_eps_estimate}.
\end{proof}

We introduce the following function in  $\R^3$: 
\begin{equation}
\label{phi_eps}
\bphi_\eps(\x):= \eps \theta_\eps(\x) \Lambda^\eps(\x) \left( S_\eps b(\D) \wt{\f}_0 \right)(\x).
\end{equation}

\begin{lemma}\label{lem4.2}
Let $\bphi_\eps$ be defined by \eqref{phi_eps}.
For $0< \eps \le \eps_0$ we have
\begin{equation}\label{s_eps_H1}
\| \s_\eps \|_{H^1(\O)} \le {\mathcal C}_8 \Big( \eps^{1/2} \| \FF \|_{L_2(\O)} + \| \bphi_\eps\|_{H^1(\R^3)} \Big),
\end{equation}
where the constant ${\mathcal C}_8$ depends only on the problem data \eqref{data}.
\end{lemma}

\begin{proof}
By \eqref{popravka}, \eqref{popravka_identity}, and \eqref{srezka}, the function
$\s_\eps - \bphi_\eps \in H^1(\O;\C^3)$ satisfies the boundary condition  
$$
\big(\mu_0^{1/2} (\s_\eps - \bphi_\eps)\big)_n\vert_{\partial \O}=0
$$
 and the identity 
\begin{equation}
\label{s_eps - phi_eps}
\begin{split}
l_\eps[ \s_\eps -\bphi_\eps, \bzeta] + (\s_\eps -\bphi_\eps, \bzeta)_{L_2(\O)} 
= {\mathcal I}_\eps[\bzeta] - {\mathcal J}_\eps[\bzeta],
\\
\bzeta \in H^1(\O;\C^3), \quad (\mu_0^{1/2} \bzeta)_n\vert_{\partial \O}=0,
\end{split}
\end{equation}
where  
\begin{equation}\label{J_eps}
{\mathcal J}_\eps[\bzeta] := 
(g^\eps b(\D) \bphi_\eps, b(\D) \bzeta)_{L_2(\O)}  + (\bphi_\eps, \bzeta)_{L_2(\O)}.
\end{equation}
By \eqref{form_l_eps_est}, we have 
\begin{equation}\label{J_eps_est}
|{\mathcal J}_\eps[\bzeta] | \le \sqrt{c_2} \| \D \bphi_\eps \|_{L_2(\R^3)} (l_\eps[\bzeta, \bzeta])^{1/2} + 
\| \bphi_\eps \|_{L_2(\R^3)} \|\bzeta\|_{L_2(\O)}.
\end{equation}

Substituting $\bzeta= \s_\eps -\bphi_\eps$ in \eqref{s_eps - phi_eps} and using \eqref{I_eps_estimate} and \eqref{J_eps_est}, we arrive at
\begin{align*}
l_\eps[ \s_\eps -\bphi_\eps, \s_\eps -\bphi_\eps ] + \|\s_\eps -\bphi_\eps\|^2_{L_2(\O)}
\le
2 {\mathcal C}^2_7 \eps \| \FF\|^2_{L_2(\O)} + 2 c_2 \| \D \bphi_\eps \|^2_{L_2(\R^3)} + \| \bphi_\eps \|^2_{L_2(\R^3)}.
\end{align*}
Combining this with the lower estimate \eqref{form_l_eps_est_O},  we obtain 
$$
\|\s_\eps -\bphi_\eps \|_{H^1(\O)} \le \check{\mathcal C}_8 \big( \eps^{1/2} \| \FF\|_{L_2(\O)} + \|\bphi_\eps\|_{H^1(\R^3)}\big),
$$ 
where the constant $\check{\mathcal C}_8$ depends only on the problem data \eqref{data}.  This implies \eqref{s_eps_H1}.
\end{proof}

\begin{lemma}\label{lemma_phi}
Suppose that the number $\eps_1$ satisfies Condition \textnormal{\ref{condition varepsilon}}. 
Let $\bphi_\eps$ be defined by \eqref{phi_eps}. For $0< \eps \le \eps_1$ we have
\begin{align}\label{lemma_phi_0}
&\| {\bphi}_\eps  \|_{L_2(\R^3)} \le {\mathcal C}_9 \eps \| \FF\|_{L_2(\O)},
\\
\label{lemma_phi_1}
&\| {\bphi}_\eps  \|_{H^1(\R^3)} \le {\mathcal C}_{10} \eps^{1/2} \| \FF\|_{L_2(\O)},
\end{align}
where the constants ${\mathcal C}_9$ and ${\mathcal C}_{10}$ depend on the problem data \eqref{data}.  
\end{lemma}

\begin{proof}
Estimate \eqref{lemma_phi_0} follows from \eqref{3.25a} and \eqref{srezka}.

Consider the derivatives 
\begin{equation}\label{lemma_phi_3}
D_j \bphi_\eps = \eps (D_j \theta_\eps) \Lambda^\eps S_\eps b(\D) \wt{\f}_0+ \theta_\eps (D_j\Lambda)^\eps S_\eps b(\D) \wt{\f}_0
+ \eps \theta_\eps \Lambda^\eps S_\eps D_j b(\D) \wt{\f}_0.
\end{equation}
The norm of the first summand on the right is estimated with the help of \eqref{srezka} and Lemma~\ref{Lemma 3.6 from Su15}:
\begin{align*}
\eps \|(D_j \theta_\eps) \Lambda^\eps S_\eps b(\D) \wt{\f}_0\|_{L_2(\R^3)}& \le 
\kappa \| \Lambda^\eps S_\eps b(\D) \wt{\f}_0\|_{L_2((\partial \O)_\eps)}
\\
&\le \kappa {\mathfrak C}_\Lambda \sqrt{\beta_*} \eps^{1/2} \| b(\D) \wt{\f}_0\|^{1/2}_{H^1(\R^3)} \| b(\D) \wt{\f}_0\|^{1/2}_{L_2(\R^3)}
\\& \le {\mathcal C}'_{10} \eps^{1/2}\|\FF\|_{L_2(\O)},
\end{align*}
where ${\mathcal C}_{10}'= \kappa {\mathfrak C}_\Lambda \sqrt{\beta_* \alpha_1} \wh{c} C_\O^{(2)}$. 
We have taken  \eqref{bb_eps}, \eqref{Lambda_est}, and  \eqref{phi_tilde_2} into account. 
Similarly,  Lemma \ref{Lemma 3.6 from Su15} and relations \eqref{bb_eps}, \eqref{Lambda_est},  \eqref{phi_tilde_2} 
imply the following estimate for the norm of the second term in \eqref{lemma_phi_3}:
$$
\| \theta_\eps (D_j\Lambda)^\eps S_\eps b(\D) \wt{\f}_0 \|_{L_2(\R^3)} \le {\mathcal C}''_{10} \eps^{1/2}\|\FF\|_{L_2(\O)},
$$
where 
$
{\mathcal C}_{10}''= {\mathfrak C}_\Lambda \sqrt{\beta_* \alpha_1} \wh{c} C_\O^{(2)}.
$
The norm of the third term in \eqref{lemma_phi_3} is estimated by Proposition \ref{prop f^eps S_eps}
and relations \eqref{bb_eps}, \eqref{Lambda_est}, and \eqref{phi_tilde_2}:
$$
\eps \|\theta_\eps \Lambda^\eps S_\eps D_j b(\D) \wt{\f}_0\|_{L_2(\R^3)} \le {\mathcal C}_{10}''' \eps \|\FF\|_{L_2(\O)},
$$
where
$
{\mathcal C}_{10}'''= {\mathfrak C}_\Lambda \sqrt{\alpha_1} \wh{c} C_\O^{(2)}.
$
As a result, we arrive at the estimate 
$$
\| \D \bphi_\eps \|_{L_2(\R^3)} \le \check{\mathcal C}_{10} \eps^{1/2} \| \FF \|_{L_2(\O)},
$$
where the constant $\check{\mathcal C}_{10}$ depends only on the problem data \eqref{data}. 
Together with \eqref{lemma_phi_0}, this implies \eqref{lemma_phi_1}.
\end{proof}

Lemmas~\ref{lem4.2} and \ref{lemma_phi} directly imply the following statement. 

\begin{corollary}
For $0< \eps \le \eps_1$ we have 
\begin{equation}\label{s_eps_est}
\| \s_\eps \|_{H^1(\O)} \le {\mathcal C}_{11} \eps^{1/2} \| \FF \|_{L_2(\O)},
\end{equation}
where the constant ${\mathcal C}_{11}$ depends only on the problem data \eqref{data}.
\end{corollary}

\noindent\textit{Completion of the proof of Theorem} \ref{th2_O}.
Relations \eqref{vyvod1} and \eqref{s_eps_est} imply the required estimate \eqref{res3} with the constant 
${\mathcal C}_2= {\mathcal C}_5 + {\mathcal C}_{11}$. $\square$

\subsection{Proof of Theorem \ref{th3_O}}
From \eqref{res3} it follows that 
\begin{equation}\label{4.21}
\| g^\eps b(\D) (\f_\eps - \bpsi_\eps)\|_{L_2(\O)} \le {\mathcal C}_{12} \eps^{1/2} \| \FF \|_{L_2(\O)},
\end{equation}
where the constant ${\mathcal C}_{12}$ depends only on the problem data \eqref{data}. 
According to \eqref{corr2_O}, 
\begin{align*}
g^\eps b(\D) \bpsi_\eps& \!=\! g^\eps b(\D) \f_0\! +\!
\eps g^\eps b(\D) (\Lambda^\eps S_\eps b(\D) \wt{\f}_0) 
\\
&\!=\! g^\eps b(\D) \f_0\! +\!
 g^\eps (b(\D) \Lambda)^\eps S_\eps b(\D) \wt{\f}_0 
\!+ \!\eps \!\sum_{j=1}^3 g^\eps b_j \Lambda^\eps S_\eps D_j b(\D) \wt{\f}_0
\\
&\!=\! \wt{g}^\eps b(\D) \f_0 \!+\!
 g^\eps (b(\D) \Lambda)^\eps (S_\eps\!-\!I) b(\D) \wt{\f}_0 
\!+\! \eps\! \sum_{j=1}^3 g^\eps b_j \Lambda^\eps S_\eps D_j b(\D) \wt{\f}_0.
\end{align*}
By Proposition \ref{prop f^eps S_eps} and relations \eqref{Lambda_est} and \eqref{phi_tilde_2}, 
the norm of the third summand on the right does not exceed  
$C \eps \| \FF \|_{L_2(\O)}$. The second term on the right can be written as 
\begin{align*}
&g^\eps (b(\D) \Lambda)^\eps (S_\eps-I) b(\D) \wt{\f}_0
\\
&= -i \begin{pmatrix} 
\left(\Sigma^\eps + (\eta^0)^{-1} - (\eta^\eps)^{-1}\right) (S_\eps-I) \rot \mu_0^{-1/2} \wt{\f}_0
\cr 
 (\underline{\nu}- \nu^\eps) (S_\eps-I) \div \mu_0^{1/2} \wt{\f}_0
\end{pmatrix}.
\end{align*}
Similarly to the proof of Lemma \ref{lem1.7}, using Propositions \ref{prop_Seps - I} and \ref{prop_PSu}, it is easy to check that 
$$
\| g^\eps (b(\D) \Lambda)^\eps (S_\eps-I) b(\D) \wt{\f}_0\|_{L_2(\R^3)} \le {\mathcal C}'_{12} \eps \|\FF\|_{L_2(\O)},
$$
where the constant ${\mathcal C}'_{12}$ depends only on the problem data \eqref{data}.
As a result, we obtain 
\begin{equation}\label{4.23}
\| g^\eps b(\D) \bpsi_\eps - \wt{g}^\eps b(\D) \f_0\|_{L_2(\O)} \le \check{\mathcal C}_{12} \eps \| \FF \|_{L_2(\O)},
\end{equation}
where the constant $\check{\mathcal C}_{12}$ depends only on the problem data \eqref{data}.

Relations \eqref{4.21} and \eqref{4.23} imply the required estimate
$$
\| g^\eps b(\D) \f_\eps - \wt{g}^\eps b(\D) \f_0\|_{L_2(\O)} \le ({\mathcal C}_{12}+ \check{\mathcal C}_{12}) \eps^{1/2} \| \FF \|_{L_2(\O)},
$$
which is equivalent to the pair of inequalities \eqref{res5}.
$\square$

\subsection{Estimate for the correction term in $L_2(\O)$. Completion of the proof of Theorem~\ref{th1_O}}

\begin{lemma}
For $0< \eps \le \eps_1$ we have
\begin{equation}\label{4.25}
\| \s_\eps \|_{L_2(\O)} \le {\mathcal C}_{13} \eps \| \FF \|_{L_2(\O)},
\end{equation}
where the constant ${\mathcal C}_{13}$ depends only on the problem data \eqref{data}.
\end{lemma}

\begin{proof}
Let $\G \in L_2(\O;\C^3)$. We put $\bzeta_\eps := (L_\eps + I)^{-1} \G$.
We substitute $\bzeta = \bzeta_\eps$ in the identity \eqref{s_eps - phi_eps}.
Then the left-hand side of this identity takes the form  $(\s_\eps -\bphi_\eps, \G)_{L_2(\O)}$.  Hence, 
\begin{equation}
\label{4.26}
(\s_\eps -\bphi_\eps, \G)_{L_2(\O)} 
= {\mathcal I}_\eps[\bzeta_\eps] - {\mathcal J}_\eps[\bzeta_\eps].
\end{equation}
Combining \eqref{I_eps=sum}, \eqref{I2_eps_est}, \eqref{J_eps}, \eqref{lemma_phi_0}, \eqref{4.26}, and the obvious estimate 
$$
l_\eps [\bzeta_\eps, \bzeta_\eps] + \|\bzeta_\eps \|^2_{L_2(\O)} \le \|\G\|^2_{L_2(\O)},
$$
we have
\begin{equation}
\label{4.27}
|(\s_\eps -\bphi_\eps, \G)_{L_2(\O)}| \le ({\mathcal C}_{7}'+{\mathcal C}_{9}) \eps \|\FF\|_{L_2(\O)} \|\G \|_{L_2(\O)} 
+ |{\mathcal I}^{(1)}_\eps[\bzeta_\eps]| + 
|(g^\eps b(\D) \bphi_\eps, b(\D) \bzeta_\eps)_{L_2(\O)}|.
\end{equation}
Since the functions  ${\mathbf f}_{j,\eps}$ and $\bphi_\eps$ are supported in the $\eps$-neighborhood of the boundary $\partial \O$ (see \eqref{srezka}, \eqref{4.8a}, and \eqref{phi_eps}), from \eqref{I1_eps=}, \eqref{fj_eps_est}, and \eqref{lemma_phi_1} it follows that 
\begin{equation}
\label{4.28}
\begin{split}
|{\mathcal I}^{(1)}_\eps[\bzeta_\eps]| + 
|(g^\eps b(\D) \bphi_\eps, b(\D) \bzeta_\eps)_{L_2(\O)}|\le
{\mathcal C}_{13}' \eps^{1/2} \| \FF \|_{L_2(\O)} \| \D \bzeta_\eps \|_{L_2(B_\eps)},
\end{split}
\end{equation}
where the constant ${\mathcal C}_{13}'$ depends only on the problem data \eqref{data}. 

Applying the already proved Theorem \ref{th2_O}, we approximate the function~$\bzeta_\eps$ by  
$
\bzeta_0 + \eps \Lambda^\eps S_\eps b(\D) \wt{\bzeta}_0,
$
where
$\bzeta_0 = (L^0 +I)^{-1}\G$ and $\wt{\bzeta}_0 = P_\O \bzeta_0$.
We have:
\begin{equation}
\label{4.29}
\begin{split}
 \| \D \bzeta_\eps \|_{L_2(B_\eps)} &\le \| \D (\bzeta_\eps - \bzeta_0 - \eps \Lambda^\eps S_\eps b(\D) \wt{\bzeta}_0)\|_{L_2(\O)}
+ \|\D \bzeta_0 \|_{L_2(B_\eps)} 
\\
&
+ \eps \| \D ( \Lambda^\eps S_\eps b(\D) \wt{\bzeta}_0)\|_{L_2((\partial \O)_\eps)}.
\end{split}
\end{equation}
By Theorem \ref{th2_O}, the first term on the right does not exceed ${\mathcal C}_2 \eps^{1/2} \| \G \|_{L_2(\O)}$.
The second term is estimated by $\sqrt{\beta}\wh{c} \eps^{1/2}\|\G\|_{L_2(\O)}$, due to Lemma \ref{lemma ots int O_eps B_eps} and estimate \eqref{L0_H2}. Let us estimate the third term:
\begin{align*}
&\eps \| \D ( \Lambda^\eps S_\eps b(\D) \wt{\bzeta}_0)\|_{L_2((\partial \O)_\eps)} 
\\
&\le
\| (\D \Lambda)^\eps S_\eps b(\D) \wt{\bzeta}_0 \|_{L_2((\partial \O)_\eps)}+
\eps \| \Lambda^\eps S_\eps \D b(\D) \wt{\bzeta}_0\|_{L_2(\R^3)}
\\
&\le
\sqrt{\beta_*} {\mathfrak C}_\Lambda \eps^{1/2} \| b(\D) \wt{\bzeta}_0\|^{1/2}_{H^1(\R^3)} \| b(\D) \wt{\bzeta}_0 \|^{1/2}_{L_2(\R^3)}
+ {\mathfrak C}_\Lambda \eps \| \D b(\D) \wt{\bzeta}_0 \|_{L_2(\R^3)}.
\end{align*}
We have used Lemma \ref{Lemma 3.6 from Su15}, Proposition \ref{prop f^eps S_eps}, and estimate \eqref{Lambda_est}. 
Combining this with the analog of estimate \eqref{phi_tilde_2} for $\wt{\bzeta}_0$, we see that the third term in  
\eqref{4.29} does not exceed ${\mathcal C}''_{13} \eps^{1/2} \| \G \|_{L_2(\O)}$, 
where ${\mathcal C}''_{13}$ depends only on the problem data \eqref{data}. As a result, we arrive at the inequality
\begin{equation}
\label{4.30}
 \| \D \bzeta_\eps \|_{L_2(B_\eps)} \le ({\mathcal C}_{2} +\sqrt{\beta}\wh{c}+ {\mathcal C}''_{13}) \eps^{1/2}\|\G \|_{L_2(\O)},\quad 0< \eps \le \eps_1.
\end{equation}

Relations \eqref{4.27}, \eqref{4.28}, and \eqref{4.30} imply that
$$
|(\s_\eps -\bphi_\eps, \G)_{L_2(\O)}| \le
\check{\mathcal C}_{13} \eps \|\FF\|_{L_2(\O)} \|\G \|_{L_2(\O)}, \quad \forall \G \in L_2(\O;\C^3),
$$
where the constant $\check{\mathcal C}_{13}$ depends only on the problem data \eqref{data}. Hence,
$$
\| \s_\eps -\bphi_\eps \|_{L_2(\O)} \le \check{\mathcal C}_{13} \eps \|\FF\|_{L_2(\O)}.
$$
Together with estimate \eqref{lemma_phi_0}, this implies \eqref{4.25}.
\end{proof}

\noindent\textit{Completion of the proof of Theorem} \ref{th1_O}.
Relations \eqref{vyvod3} and \eqref{4.25} imply the required estimate \eqref{res1} with the constant ${\mathcal C}_1= {\mathcal C}_6 + {\mathcal C}_{13}$. $\square$

\section{The stationary Maxwell system\label{Sec5}}

\subsection{Functional classes}
As above, we assume that  ${\O \subset \R^3}$ is a bounded domain of class $C^{1,1}$. 
Recall the following definitions; see \cite{BS1, BS2}.

\begin{definition}\label{def1}
Let $\u \in L_2(\O;\C^3)$ and $\div \u \in L_2(\O)$. Then, by definition,  the relation $\u_n \vert_{\partial \O}=0$ means that
$$
(\u, \nabla \omega)_{L_2(\O)} = - (\div \u, \omega)_{L_2(\O)},\quad \forall \omega \in H^1(\O). 
$$
\end{definition}

\begin{definition}\label{def2}
Let $\u \in L_2(\O;\C^3)$ and  $\rot \u \in L_2(\O;\C^3)$. Then, by definition, the relation $\u_\tau \vert_{\partial \O}=0$ means that 
$$
(\u, \rot \z)_{L_2(\O)} = (\rot \u, \z)_{L_2(\O)},\quad \forall \z \in L_2(\O;\C^3): \ \rot \z \in L_2(\O;\C^3). 
$$
\end{definition}

Suppose that the matrix $\mu_0$ and the matrix-valued function  $\eta(\x)$ satisfy the assumptions of Subsection~1.3. 
Along with the ordinary space $L_2(\O;\C^3)$, we need to define the weighted $L_2$-spaces of vector-valued functions: the space 
$$
L_2(\O;(\eta^\eps)^{-1}) = L_2(\O;\C^3;(\eta^\eps)^{-1})
$$
with the inner product 
$$
({\mathbf f}_1, {\mathbf f}_2)_{L_2(\O;(\eta^\eps)^{-1})} = \intop_\O \langle (\eta^\eps(\x))^{-1} {\mathbf f}_1(\x), {\mathbf f}_2(\x) \rangle \,d\x
$$
and the similar space  
$$
L_2(\O;\mu_0^{-1}) = L_2(\O;\C^3;\mu_0^{-1})
$$
 with the inner product 
$$
({\mathbf f}_1, {\mathbf f}_2)_{L_2(\O;\mu_0^{-1})} = \intop_\O \langle \mu_0^{-1} {\mathbf f}_1(\x), {\mathbf f}_2(\x) \rangle \,d\x.
$$

We introduce two subspaces of divergence-free vector-valued functions in $L_2$:
\begin{align}
\label{5.1}
J(\O) :=& \bigl\{ {\mathbf u} \in L_2(\O;\C^3): \ \int_\O \langle \u, \nabla \omega \rangle \, d\x =0, \ \forall \omega \in H^1_0(\O) \bigr\},
\\
\label{5.2}
J_0(\O) :=& \bigl\{ {\mathbf u} \in L_2(\O;\C^3): \ \int_\O \langle \u, \nabla \omega \rangle \, d\x =0, \ \forall \omega \in H^1(\O) \bigr\}.
\end{align}
The subspace \eqref{5.1} consists of all functions ${\mathbf u} \in L_2(\O;\C^3)$ such that $\div \u =0$ in the sense of distributions. 
The subspace \eqref{5.2} consists of all functions ${\mathbf u} \in L_2(\O;\C^3)$ such that $\div \u =0$ and  
$\u_n\vert_{\partial \O}=0$ (in the sense of Definition \ref{def1}).

\subsection{Statement of the problem}

We study an electromagnetic resonator filling the domain~$\O$.
Suppose that the magnetic permeability is given by the constant matrix $\mu_0$, and the dielectric permittivity is given by the matrix  
$\eta^\eps(\x) = \eta(\eps^{-1}\x)$.
The intensities of the electric and magnetic fields are denoted by $\u_\eps(\x)$ and $\v_\eps(\x)$, respectively.
The electric and magnetic displacement vectors are expressed in terms of $\u_\eps$, $\v_\eps$ by  
$\w_\eps(\x) = \eta^\eps(\x) \u_\eps(\x)$, $\z_\eps(\x)= \mu_0 \v_\eps(\x)$. 

The operator $M_\eps$ written in terms of the displacement vectors acts in the space  $J(\O) \oplus J_0(\O)$ considered as a subspace of
$$
L_2\big(\O;\C^3; (\eta^\eps)^{-1}\big) \oplus L_2\big(\O;\C^3;\mu_0^{-1}\big),
$$
 and is given by 
\begin{equation}\label{M_eps}
M_\eps = \begin{pmatrix} 0 & i \rot \mu_0^{-1} \cr -i \rot (\eta^\eps)^{-1} & 0 \end{pmatrix}
\end{equation}
on the domain 
\begin{equation}\label{Dom_M_eps}
\begin{aligned}
\Dom M_\eps = &\big\{ (\w,\z) \in J(\O) \oplus J_0(\O): \ \rot (\eta^\eps)^{-1} \w \in L_2(\O;\C^3),
\\ 
&\rot \mu_0^{-1} \z \in L_2(\O;\C^3),
\ ((\eta^\eps)^{-1} \w)_\tau \vert_{\partial \O}=0 \big\}. 
\end{aligned}
\end{equation}
Here the boundary condition for $\w$ is understood in the sense of Definition \ref{def2}.

The operator $M_\eps$ is selfadjoint; see \cite{BS1,BS2}. The point $\lambda =i$ is a regular point for the operator $M_\eps$. 
\textit{Our goal} is to study the behavior of the resolvent $(M_\eps - i I)^{-1}$.
In other words, we are interested in the behavior of the solutions $(\w_\eps,\z_\eps)$ of the equation  
\begin{equation}\label{M1}
(M_\eps - iI) \begin{pmatrix} \w_\eps \cr \z_\eps \end{pmatrix} = \begin{pmatrix} \q \cr \r \end{pmatrix}, \quad \q \in J(\O),\ \r \in J_0(\O),
\end{equation}
and also in the behavior of the fields $\u_\eps = (\eta^\eps)^{-1}\w_\eps$ and $\v_\eps = \mu_0^{-1}\z_\eps$.  
In details, the Maxwell system \eqref{M1} takes the form 
\begin{equation}\label{M02}
\left\{\begin{matrix}
& i \rot \mu_0^{-1} \z_\eps - i \w_\eps = \q,
\cr
& -i \rot (\eta^\eps)^{-1} \w_\eps - i \z_\eps = \r,
\cr
& \div \w_\eps=0,\ \div \z_\eps =0,
\cr
& ((\eta^\eps)^{-1} \w_\eps)_\tau\vert_{\partial \O} =0,\  (\z_\eps)_n \vert_{\partial \O} =0.
\end{matrix}
\right.
\end{equation}

Let $\eta^0$ be the effective matrix defined by  \eqref{v_eq6} and \eqref{eta0}. Let $M^0$ be the effective Maxwell operator 
with the coefficients $\eta^0$ and $\mu_0$ (defined similarly to \eqref{M_eps} and \eqref{Dom_M_eps}).
Consider the homogenized equation
\begin{equation}\label{M0}
(M^0 - iI) \begin{pmatrix} \w_0 \cr \z_0 \end{pmatrix} = \begin{pmatrix} \q \cr \r \end{pmatrix}, 
\end{equation}
and define the functions $\u_0 = (\eta^0)^{-1} \w_0$ and $\v_0 = \mu_0^{-1} \z_0$.
In details, \eqref{M0} takes the form  
\begin{equation}\label{M2}
\left\{\begin{matrix}
& i \rot \mu_0^{-1} \z_0 - i \w_0 = \q,
\cr
& -i \rot (\eta^0)^{-1} \w_0 - i \z_0 = \r,
\cr
& \div \w_0=0,\ \div \z_0 =0,
\cr
& ((\eta^0)^{-1} \w_0)_\tau\vert_{\partial \O} =0,\  (\z_0)_n \vert_{\partial \O} =0.
\end{matrix}
\right.
\end{equation}

The classical results (see \cite{BeLPap,  Sa, ZhKO}) show that \textit{the fields $\u_\eps$, $\w_\eps$
$\v_\eps$, $\z_\eps$ weakly converge in $L_2(\O;\C^3)$ to the corresponding homogenized fields $\u_0$, $\w_0$, $\v_0$, $\z_0$, as $\eps \to 0$}.

\subsection{The case where $\q=0$. Reduction of the problem to the model second order equation}
If $\q=0$, the system \eqref{M02} takes the form 
\begin{equation}\label{M3}
\left\{\begin{matrix}
& \w_\eps = \rot \mu_0^{-1} \z_\eps,
\cr
& \rot (\eta^\eps)^{-1} \w_\eps + \z_\eps = i \r,
\cr
& \div \w_\eps=0,\ \div \z_\eps =0,
\cr
& ((\eta^\eps)^{-1} \w_\eps)_\tau\vert_{\partial \O} =0,\  (\z_\eps)_n \vert_{\partial \O} =0.
\end{matrix}
\right.
\end{equation}
From \eqref{M3} it follows that $\z_\eps$ is the solution of the problem
\begin{equation*}
\left\{\begin{matrix}
& \rot (\eta^\eps)^{-1} \rot \mu_0^{-1} \z_\eps + \z_\eps = i \r, \quad \div \z_\eps =0,
\cr
& (\z_\eps)_n \vert_{\partial \O} =0, \quad ((\eta^\eps)^{-1} \rot \mu_0^{-1}\z_\eps)_\tau\vert_{\partial \O} =0.
\end{matrix}
\right.
\end{equation*}
Then the function $\f_\eps:= \mu_0^{-1/2} \z_\eps$ is the solution of the problem  
\begin{equation}\label{M5}
\left\{\begin{matrix}
& \mu_0^{-1/2} \rot (\eta^\eps)^{-1} \rot \mu_0^{-1/2} \f_\eps + \f_\eps = i \mu_0^{-1/2} \r, \quad \div \mu_0^{1/2}\f_\eps =0,
\cr
& (\mu_0^{1/2}\f_\eps)_n \vert_{\partial \O} =0, \quad ((\eta^\eps)^{-1} \rot \mu_0^{-1/2}\f_\eps)_\tau\vert_{\partial \O} =0.
\end{matrix}
\right.
\end{equation}
Obviously, the solution of problem \eqref{M5} is simultaneously the solution of the problem  
\begin{equation}\label{M6}
\left\{\begin{matrix}
& \mu_0^{-1/2} \rot (\eta^\eps)^{-1} \rot \mu_0^{-1/2} \f_\eps - \mu_0^{1/2} \nabla \div \mu_0^{1/2} \f_\eps+ \f_\eps = i \mu_0^{-1/2} \r, 
\cr
& (\mu_0^{1/2}\f_\eps)_n \vert_{\partial \O} =0, \quad ((\eta^\eps)^{-1} \rot \mu_0^{-1/2}\f_\eps)_\tau\vert_{\partial \O} =0.
\end{matrix}
\right.
\end{equation}
(Note that the condition $\r \in J_0(\O;\C^3)$ automatically implies that \hbox{$\div \mu_0^{1/2}\f_\eps =0$}.)
The problem \eqref{M6} coincides with \eqref{phi_eps_pr} if $\nu=1$ and $\FF = i \mu_0^{-1/2} \r$.

Let $L_\eps$ be the operator defined in Subsection \ref{Sec2.2} with the coefficients $\mu_0$, $\eta^\eps$, and $\nu=1$. 
We see that the solution $\f_\eps$ of problem \eqref{M5} can be written as $\f_\eps = i (L_\eps +I)^{-1} (\mu_0^{-1/2} \r)$. 

Similarly, in the case where $\q=0$, the effective system \eqref{M2} takes the form
\begin{equation}\label{M7}
\left\{\begin{matrix}
& \w_0 = \rot \mu_0^{-1} \z_0,
\cr
& \rot (\eta^0)^{-1} \w_0 + \z_0 = i \r,
\cr
& \div \w_0 =0,\ \div \z_0 =0,
\cr
& ((\eta^0)^{-1} \w_0)_\tau\vert_{\partial \O} =0,\  (\z_0)_n \vert_{\partial \O} =0.
\end{matrix}
\right.
\end{equation}
Then $\z_0$ is the solution of the problem 
\begin{equation*}
\left\{\begin{matrix}
& \rot (\eta^0)^{-1} \rot \mu_0^{-1} \z_0 + \z_0 = i \r, \quad \div \z_0 =0,
\cr
& (\z_0)_n \vert_{\partial \O} =0, \quad ((\eta^0)^{-1} \rot \mu_0^{-1}\z_0)_\tau\vert_{\partial \O} =0.
\end{matrix}
\right.
\end{equation*}
Hence, the function $\f_0:= \mu_0^{-1/2} \z_0$ is the solution of the problem
\begin{equation}\label{M9}
\left\{\begin{matrix}
& \mu_0^{-1/2} \rot (\eta^0)^{-1} \rot \mu_0^{-1/2} \f_0 + \f_0 = i \mu_0^{-1/2} \r, \quad \div \mu_0^{1/2}\f_0 =0,
\cr
& (\mu_0^{1/2}\f_0)_n \vert_{\partial \O} =0, \quad ((\eta^0)^{-1} \rot \mu_0^{-1/2}\f_0)_\tau\vert_{\partial \O} =0.
\end{matrix}
\right.
\end{equation}
Clearly, the solution of problem \eqref{M9} is simultaneously the solution of the problem
\begin{equation}\label{M10}
\left\{\begin{matrix}
& \mu_0^{-1/2} \rot (\eta^0)^{-1} \rot \mu_0^{-1/2} \f_0 - \mu_0^{1/2} \nabla \div \mu_0^{1/2} \f_0+ \f_0 = i \mu_0^{-1/2} \r, 
\cr
& (\mu_0^{1/2}\f_0)_n \vert_{\partial \O} =0, \quad ((\eta^0)^{-1} \rot \mu_0^{-1/2}\f_0)_\tau\vert_{\partial \O} =0.
\end{matrix}
\right.
\end{equation}
(The condition $\r \in J_0(\O;\C^3)$ automatically implies that $\div \mu_0^{1/2}\f_0 =0$.)
The problem \eqref{M10} coincides with  \eqref{phi_0_pr}, if $\underline{\nu}=1$ and $\FF = i \mu_0^{-1/2} \r$.

Let $L^0$ be the effective operator defined in Subsection \ref{Sec2.3} with the coefficients $\mu_0$, $\eta^0$, and
 $\underline{\nu} =1$.
Then the solution $\f_0$ of  problem \eqref{M9} can be written as $\f_0= i (L^0+I)^{-1}(\mu_0^{-1/2} \r)$.

\subsection{Results for the Maxwell system}

Applying Theorem \ref{th1_O} and using the relations  
$$
\z_\eps = \mu_0^{1/2} \f_\eps,\quad \v_\eps = \mu_0^{-1/2} \f_\eps,\quad \z_0 = \mu_0^{1/2} \f_0,\quad \v_0 = \mu_0^{-1/2} \f_0,\quad \FF = i \mu_0^{-1/2}\r,
$$
for $0<\eps \le \eps_1$ we obtain 
\begin{align}
\label{MM1}
&\| \z_\eps - \z_0\|_{L_2(\O)} \le |\mu_0|^{1/2} \| \f_\eps - \f_0\|_{L_2(\O)} \le {\mathcal C}_1 |\mu_0|^{1/2}|\mu_0^{-1}|^{1/2} \eps \|\r\|_{L_2(\O)}, 
\\
\label{MM2}
&\| \v_\eps - \v_0\|_{L_2(\O)} \le |\mu_0^{-1}|^{1/2} \| \f_\eps - \f_0\|_{L_2(\O)} \le {\mathcal C}_1 |\mu_0^{-1}| \eps \|\r\|_{L_2(\O)}. 
\end{align}

Now we apply Theorem~\ref{th2_O}. If $\nu=1$, the solution of equation~\eqref{v_eq13} is equal to zero: $\rho(\x)=0$, whence the function  \eqref{corr2_O} takes the form  
$$
\bpsi_\eps = \f_0 + \eps \mu_0^{-1/2} \Psi^\eps S_\eps \rot \mu_0^{-1/2} \wt{\f}_0.
$$
Denote $\wt{\w}_0 = \rot \mu_0^{-1/2} \wt{\f}_0$. Clearly, $\wt{\w}_0$ is an extension of the function $\w_0 = \rot \mu_0^{-1/2} \f_0$. 
From \eqref{res3} it follows that for $0< \eps \le \eps_1$ we have
\begin{align}
\label{MM3}
&\| \z_\eps - \z_0 - \eps \Psi^\eps S_\eps \wt{\w}_0 \|_{H^1(\O)} 
\le {\mathcal C}_2 |\mu_0|^{1/2}|\mu_0^{-1}|^{1/2} \eps^{1/2} \|\r\|_{L_2(\O)}, 
\\
\label{MM4}
&\| \v_\eps - \v_0 - \eps \mu_0^{-1} \Psi^\eps S_\eps \wt{\w}_0 \|_{H^1(\O)} \le  {\mathcal C}_2 |\mu_0^{-1}| \eps^{1/2} \|\r\|_{L_2(\O)}. 
\end{align}

Next, the first equation in \eqref{M3} implies that  $\w_\eps = \rot \mu_0^{-1/2} \f_\eps$, 
whence $\u_\eps = (\eta^\eps)^{-1} \rot \mu_0^{-1/2} \f_\eps$.
Similarly, $\w_0 = \rot \mu_0^{-1/2} \f_0$  and $\u_0 = (\eta^0)^{-1} \rot \mu_0^{-1/2} \f_0$.
Applying Theorem \ref{th3_O}, we see that 
\begin{equation}
\label{MM5}
\| \u_\eps - \u_0 - \Sigma^\eps {\w}_0 \|_{L_2(\O)} 
\le {\mathcal C}_3 |\mu_0^{-1}|^{1/2} \eps^{1/2} \|\r\|_{L_2(\O)} 
\end{equation}
for $0< \eps \le \eps_1$. Recalling that $\Sigma(\x)$ is the matrix with the columns $\nabla \Phi_j(\x)$, $j=1,2,3,$
where $\Phi_j(\x)$ is the \hbox{$\Gamma$-periodic} solution of problem \eqref{1.34a}, 
we represent this matrix in the form $\Sigma(\x)= \Xi(\x) (\eta^0)^{-1}$, where $\Xi(\x)$ is the matrix with the columns  
$\nabla Y_j(\x)$, $j=1,2,3$, and $Y_j(\x)$ is the \hbox{$\Gamma$-periodic} solution of the problem  
\begin{equation}
\label{MM6}
\div \eta(\x) (\nabla Y_j(\x)+ \wt{\e}_j)=0, 
\quad \intop_\Omega Y_j(\x)\, d\x=0.
\end{equation}
Since $(\eta^0)^{-1}\w_0=\u_0$, we rewrite \eqref{MM5} as 
\begin{equation}
\label{MM7}
\| \u_\eps - \u_0 - \Xi^\eps {\u}_0 \|_{L_2(\O)} 
\le {\mathcal C}_3 |\mu_0^{-1}|^{1/2} \eps^{1/2} \|\r\|_{L_2(\O)}. 
\end{equation}
Combining the relations $\w_\eps = \eta^\eps \u_\eps$, $\w_0 = \eta^0 \u_0$, and \eqref{MM7}, we deduce
\begin{equation}
\label{MM8}
\| \w_\eps - \w_0 - \Upsilon^\eps {\w}_0 \|_{L_2(\O)} 
\le {\mathcal C}_3 |\mu_0^{-1}|^{1/2} \|\eta\|_{L_\infty}\eps^{1/2} \|\r\|_{L_2(\O)},
\end{equation}
where $\Upsilon(\x) = \wt{\eta}(\x)(\eta^0)^{-1}-\1$, $\wt{\eta}(\x):= \eta(\x)(\Xi(\x)+\1)$.

Relations \eqref{MM1}--\eqref{MM4}, \eqref{MM7}, and \eqref{MM8} imply the following final result about homogenization of the solutions of the Maxwell system with $\q=0$. 

\begin{theorem}\label{th_Maxwell}
Suppose that $\O \subset \R^3$ is a bounded domain of class $C^{1,1}$.
Suppose that $\mu_0$ is a positive matrix with real entries and $\eta(\x)$ is a $\Gamma$-periodic matrix-valued function with real entries such that 
$\eta(\x)>0$ and $\eta,\eta^{-1}\in L_\infty$. Let $(\w_\eps, \z_\eps)$ be the solution of system \eqref{M3} with $\r \in J_0(\O;\C^3)$. 
Let $\u_\eps = (\eta^\eps)^{-1}\w_\eps$ and $\v_\eps = \mu_0^{-1}\z_\eps$. Suppose that $(\w_0, \z_0)$ is the solution of the homogenized system   \eqref{M7} with the effective matrix $\eta^0$ defined by \eqref{v_eq6} and \eqref{eta0}. 
Let $\u_0 = (\eta^0)^{-1}\w_0$ and $\v_0 = \mu_0^{-1}\z_0$. 
Suppose that $\eps_1$ is subject to Condition~\textnormal{\ref{condition varepsilon}}. Then the following statements hold. 

\noindent 
\emph{1)}  The fields $\v_\eps,\  \z_\eps$ converge to $\v_0,\ \z_0$, respectively, in the $L_2(\O;\C^3)$-norm, as $\eps \to 0$.
Moreover, for $0< \eps \le \eps_1$ we have
$$
\begin{aligned}
&\| \v_\eps - \v_0  \|_{L_2(\O)}  \le {\mathfrak C}_1  \eps \|\r\|_{L_2(\O)}, 
\\
&\| \z_\eps - \z_0  \|_{L_2(\O)} \le {\mathfrak C}_2  \eps \|\r\|_{L_2(\O)}.
\end{aligned}
$$

\noindent 
\emph{2)} Let $S_\eps$ be the Steklov smoothing operator defined by \eqref{S_eps}. 
Let $\Psi(\x)$ be the matrix with the columns $\rot \p_j(\x)$, $j=1,2,3,$
where $\p_j$ is the \hbox{$\Gamma$-periodic} solution of problem \eqref{v_eq11}. 
Let $\wt{\w}_0(\x)$ be the extension of the function  $\w_0(\x)$ to $\R^3$ constructed above. 
Then for $0< \eps \le \eps_1$ the fields $\v_\eps, \ \z_\eps$ satisfy approximations in the  
$H^1(\O;\C^3)$-norm with the following error estimates\emph{:}
\begin{align}
\label{MM18}
&\| \v_\eps - \v_0 - \eps \mu_0^{-1} \Psi^\eps S_\eps \wt{\w}_0 \|_{H^1(\O)}  \le {\mathfrak C}_3  \eps^{1/2} \|\r\|_{L_2(\O)}, 
\\
\label{MM19}
&\| \z_\eps - \z_0 - \eps \Psi^\eps S_\eps \wt{\w}_0 \|_{L_2(\O)} \le {\mathfrak C}_4  \eps^{1/2} \|\r\|_{L_2(\O)}.
\end{align}

\noindent 
\emph{3)}  Let $\Xi(\x)$ be the matrix with the columns $\nabla Y_j(\x)$, $j=1,2,3,$
where $Y_j$ is the \hbox{$\Gamma$-periodic} solution of problem \eqref{MM6}. 
Let $\wt{\eta}(\x):= \eta(\x)(\Xi(\x)+\1)$ and $\Upsilon(\x):= \wt{\eta}(\x)(\eta^0)^{-1} - \1$.
Then for $0< \eps \le \eps_1$ the fields $\u_\eps, \ \w_\eps$ satisfy approximations in the  
$L_2(\O;\C^3)$-norm with the following error estimates\emph{:}
\begin{align}
\label{MM20}
&\| \u_\eps - \u_0 -  \Xi^\eps  {\u}_0 \|_{L_2(\O)}  \le {\mathfrak C}_5  \eps^{1/2} \|\r\|_{L_2(\O)}, 
\\
\label{MM21}
&\| \w_\eps - \w_0 - \Upsilon^\eps {\w}_0 \|_{L_2(\O)} \le {\mathfrak C}_6  \eps^{1/2} \|\r\|_{L_2(\O)}.
\end{align}

\noindent
The constants ${\mathfrak C}_1$, ${\mathfrak C}_2$, ${\mathfrak C}_3$, ${\mathfrak C}_4$, ${\mathfrak C}_5$, and ${\mathfrak C}_6$
depend only on $|\mu_0|$, $|\mu_0^{-1}|$, $\|\eta\|_{L_\infty}$, $\|\eta^{-1}\|_{L_\infty}$, the parameters of the lattice $\Gamma$, and the domain~$\O$.
\end{theorem}

\begin{remark}
1) We see that there is no symmetry in the results for the magnetic fields $\v_\eps,\ \z_\eps$ and the electric fields $\u_\eps,\ \w_\eps$. 
The magnetic fields converge in the $L_2$-norm, with error estimates being of sharp order $O(\eps)$, and admit approximations in  
$H^1$ with the error terms of order $O(\sqrt{\eps})$. The electric fields are approximated only in $L_2$ with the error terms of order
$O(\sqrt{\eps})$. This is explained by the absence of symmetry in the statement of the problem: we assume that  
$\q=0$ in the right-hand side of system \eqref{M1}. For this reason, the function $\z_\eps$ is a solution of the auxiliary second order equation, 
while $\w_\eps$ is given in terms of the derivatives of this solution.  

 2) Note that the mean values of the periodic matrix-valued functions $\Xi(\x)$ and $\Upsilon(\x)$ are equal to zero. Therefore, by the mean value property,
the correction terms $\Xi^\eps  {\u}_0$ and $\Upsilon^\eps {\w}_0$  weakly tend to zero in $L_2$. Then relations \eqref{MM20} and \eqref{MM21} imply that the fields $\u_\eps$ and $\w_\eps$ weakly converge in $L_2$ to $\u_0$ and $\w_0$, respectively. This agrees with the classical results.
The terms $\Xi^\eps  {\u}_0$ and $\Upsilon^\eps {\w}_0$ can be interpreted as the zero order correctors.

 3) Similarly,  the correction terms $\eps \mu_0^{-1} \Psi^\eps S_\eps \wt{\w}_0$
and $\eps \Psi^\eps S_\eps \wt{\w}_0$ from \eqref{MM18}, \eqref{MM19} weakly tend to zero in $H^1$. Hence, the fields  
$\v_\eps$ and $\z_\eps$ weakly converge in $H^1$ to $\v_0$ and $\z_0$, respectively.
\end{remark}

Now, we distinguish the special cases. By Proposition \ref{prop_cases} and Remark \ref{rem_eta0}, from Theorem \ref{th_Maxwell}
we deduce the following statement.

\begin{proposition}

\noindent\emph{1)} Suppose that $\eta^0= \overline{\eta}$, i.~e., the columns of the matrix $\eta(\x)$ are divergence free.
Then for $0< \eps \le \eps_1$ we have
$$
\| \u_\eps - \u_0  \|_{L_2(\O)} \le {\mathfrak C}_5 \eps^{1/2} \|\r\|_{L_2(\O)}.
$$

\noindent\emph{2)} Suppose that $\eta^0= \underline{\eta}$, i.~e., the columns of the matrix $\eta(\x)^{-1}$ are potential.
Then for $0< \eps \le \eps_1$ we have  
$$
\begin{aligned}
\| \v_\eps - \v_0 \|_{H^1(\O)} \le \mathfrak{C}_3 \eps^{1/2} \|\r \|_{L_2(\O)},
\\
\| \z_\eps - \z_0 \|_{H^1(\O)} \le \mathfrak{C}_4 \eps^{1/2} \|\r \|_{L_2(\O)}.
\end{aligned}
$$
\end{proposition}

\end{document}